\documentclass[a4paper,10pt]{article}
\usepackage[utf8]{inputenc}
\usepackage{amsmath}
\usepackage{etex}
\usepackage{etoolbox}
\usepackage{ifmtarg}
\usepackage{xstring}
\usepackage{colortbl}
\usepackage[table,hyperref]{xcolor}
\usepackage{longtable}
\usepackage{booktabs}
\usepackage{graphicx}
\usepackage[OT2,T1]{fontenc}
\usepackage{fix-cm} 
\usepackage[french,english]{babel}
\usepackage{ucs}
\usepackage{amsmath}
\usepackage{amssymb}
\usepackage{amsfonts}
\usepackage{mathabx}
\usepackage{dsfont}
\usepackage{nicefrac}
\usepackage{units}
\usepackage{stmaryrd}
\usepackage{graphicx}
\usepackage{wrapfig}
\usepackage{placeins}
\usepackage[final]{listings}
\usepackage{caption}
\usepackage{url}
\usepackage{textcomp}
\usepackage{epsfig}
\usepackage{paralist}
\usepackage{subfig}
\usepackage{makeidx}
\usepackage{fancyhdr}
\usepackage[all,pdf]{xy}
\usepackage{picinpar}
\usepackage{enumitem}
\usepackage{mathtools}
\usepackage{arydshln}
\usepackage{mathrsfs}
\usepackage{multirow}
\usepackage{faktor}
\usepackage{relsize}

\usepackage[multiple,marginal,flushmargin,hang]{footmisc}
\usepackage[originalcommands]{ragged2e}
\usepackage{calligra}
\usepackage{xspace}
\usepackage{pdfpages}
\usepackage{tikz}
\usepackage{lastpage}
\usepackage{enumitem}
\usepackage[cal=esstix,calscaled=.96]{mathalfa}
\usepackage{amsthm}

\usepackage[retainorgcmds]{IEEEtrantools} 

\usetikzlibrary{cd}
\usetikzlibrary{matrix}
\usepgflibrary{arrows}
\tikzstyle{graphpoint}=[circle,draw=black,fill=black,inner sep=0pt,minimum size=5pt]
\tikzstyle{graphpart}=[circle,draw,minimum size=6mm]

\newcommand*{\eg}{e.\,g.\@\xspace}
\newcommand*{\ie}{i.\,e.\@\xspace}

\newcommand{\NTH}{^{\text{th}}}

\newcommand{\mc}{\mathfrak{c}}
\newcommand{\mP}{\mathbb{P}}
\newcommand{\mQ}{\mathbb{Q}}
\newcommand{\mZ}{\mathbb{Z}}
\newcommand{\mN}{\mathbb{N}}
\newcommand{\mG}{\mathbb{G}}
\newcommand{\mO}{\mathcal{O}}

\newcommand{\et}{\text{\'et}}
\newcommand{\red}{\textit{red}}

\DeclareMathOperator{\Spec}{\textit{Spec}}
\DeclareMathOperator{\coker}{\text{coker}}

\DeclareMathOperator{\Ind}{\text{Ind}}
\DeclareMathOperator{\Hom}{\text{Hom}}
\DeclareMathOperator{\Pic}{\text{Pic}}
\DeclareMathOperator{\Div}{\text{div}}

\newtheorem{definition}{Definition}[section]
\newtheorem{theorem}[definition]{Theorem}
\newtheorem{lemma}[definition]{Lemma}
\newtheorem{proposition}[definition]{Proposition}

\newtheorem{corollary}[definition]{Corollary}

\title{Aspherical neighborhoods on arithmetic surfaces: the local case}
\author{Katharina H\"ubner}

\begin{document}
\maketitle
\begin{abstract}
On arithmetic surfaces over henselian discrete valuation rings we examine
whether a geometric point has a basis of \'etale neighborhoods
 whose $\mc$-completed \'etale homotopy types are of type $K(\pi,1)$ with respect to a full class~$\mc$ of finite groups.
\end{abstract}

\section{Introduction}

In~\cite{SGA4} Artin proves the comparison theorem of classical with \'etale cohomology for a variety over the complex numbers.
A crucial point in the proof is the construction of a special type of neighborhoods now called good Artin neighborhoods.
They are open subschemes which admit a successive fibration into affine curves~$X \to B$ with smooth compactification~$\bar{X} \to B$ such that the complement~$\bar{X}-X$ is \'etale over~$B$.
This type of fibration is called elementary fibration.
The construction of a good Artin neighborhood uses Bertini's theorem in order to find a suitable linear subspace of the ambient projective space such that projection along this subspace locally yields an elementary fibration.
These neighborhoods are so useful because topologically they are particularly simple.
They are examples of $K(\pi,1)$-spaces, \ie of spaces whose only nontrivial homotopy group is the fundamental group.
This property can be drawn from the long exact homotopy sequence associated with an elementary fibration.

The scenario where~$X$ is a smooth variety over an algebraically closed field of positive characteristic was treated by Friedlander in~\cite{MR0313254}.
He examines whether an \'etale neighborhood is $K(\pi,1)$ with respect to a prime number~$\ell$ different from the characteristic of~$X$,
 \ie if it is $K(\pi,1)$ after $\ell$-completion of the \'etale homotopy type.
He also uses elementary fibrations and the homotopy sequence associated with these fibrations.
The major problem he has to deal with is non-exactness of~$\mc$-completions for a full class of finite groups~$\mc$.
If~$\mc$ is the class of finite~$\ell$-groups, he can prove that under certain conditions~$\ell$-completion is indeed exact by using special features of~$\ell$-groups.

In the arithmetic setting, \ie considering schemes flat and of finite type over~$\mZ$ or~$\mZ_p$, the above approach is not promising.
In fact, \'etale bases of neighborhoods which admit an elementary fibration never exist (see \cite{Dissertation}, Chapter~3).
In case of arithmetic surfaces~$\pi: X \to B$ this is quite obvious because for~$U \to X$ \'etale the restriction of~$\pi$ to~$U$ is the only possible fibration into curves
(unless the generic fiber is rational but in this case~$X$ may be replaced by an \'etale neighborhood).
Even if~$X$ is smooth and projective over~$B$, there are open subschemes with no \'etale neighborhood admitting an elementary fibration.
One just has to take the complement of a non-smooth divisor in~$X$.

As a consequence we cannot expect to work with smooth fibrations of arithmetic schemes.
This makes it hard to use the machinery of long exact sequences of homotopy groups associated with a fibration.
The problem is the lack of a simple relation between the homotopy theoretic fiber and the geometric fibers.
Instead, we follow a more explicit approach working directly with the Leray spectral sequence associated with a fibration.
Furthermore, we restrict our attention to arithmetic surfaces, the one-dimensional case having been dealt with in~\cite{MR2629694}.

The present work examines arithmetic surfaces which are of finite type over some local ring of integers of residue characteristic~$p$.
In contrast to Friedlander we do not need to restrict our attention to $\ell$-extensions for a single prime $\ell \neq p$ but can consider more general classes of finite groups.
We say that a noetherian scheme~$X$ is $K(\pi,1)$ with respect to a full class of finite groups~$\mc$ if the pro-$\mc$-completion of the \'etale homotopy type of~$X$ is $K(\pi,1)$.
Writing $\mN(\mc)$ for the submonoid of $\mN$ consisting of all cardinalities of groups in~$\mc$ the main result reads as follows.

\begin{theorem} \label{theoremintro}
 Let~$B$ be the spectrum  of the ring of integers of the completion~$K$ of an algebraic extension of~$\mQ_p$ with finite ramification index.
 Let~$\mc$ be a full class of finite groups such that the residue characteristic of~$B$ is not contained in~$\mN(\mc)$
  and for all but finitely many primes~$\ell \in \mN(\mc)$ the extension~$K(\mu_\ell)|K$ is a~$\mc$-extension.
 Let~$Y/B$ be an arithmetic surface and~$\bar{y} \to Y$ a geometric point.
 Then~$\bar{y}$ has a basis of \'etale neighborhoods which are~$K(\pi,1)$ with respect to~$\mc$.
\end{theorem}

In particular, there exist $K(\pi,1)$-neighborhoods with respect to any class of finite groups of the form~$\mc(\ell_1,\ldots,\ell_n)$
 for prime numbers~$\ell_i$ prime to the residue characteristic of~$B$.
Here, $\mc(\ell_1,\ldots,\ell_n)$ denotes the class of finite groups whose order is divisible at most by the primes $\ell_1,\ldots,\ell_n$.
If the residue field of~$B$ is separably closed, we can take any full class of finite groups with $p \notin \mN(\mc)$.
In particular, we may take the class of all finite groups whose order is prime to $p$.

Let us explain more closely what a $K(\pi,1)$-scheme is.
Consider a connected, locally noetherian scheme~$X$ with geometric point~$\bar{x}$.
Following~\cite{AM} we associate with $(X,\bar{x})$ the \'etale homotopy type~$X_{et}$, which is a pro-object of the homotopy category of pointed, connected CW-complexes.
We obtain homotopy pro-groups~$\pi_n(X_{\et})$ and for an abelian group~$A$ with a~$\pi_1(X_{\et})$-action cohomology groups~$H^n(X_{\et},A)$.
The first homotopy pro-group of~$X_{\et}$,~$\pi_1(X_{\et})$, coincides with the "pro-groupe fondamentale enlargi" defined in~\cite{SGA3}, Exp.~X, §6 (see~\cite{AM}, Corollary~10.7).
If~$X$ is geometrically unibranch~(\eg normal),~$\pi_1(X_{\et})$ is profinite and coincides with the usual fundamental group defined in~\cite{SGA1}, Exp.~V.
Moreover, for an abelian group~$A$ with a~$\pi_1(X_{\et})$-action the cohomology groups~$H^n(X_{\et},A)$ coincide with the \'etale cohomology groups~$H^n(X,A)$.

Let~$n$ be a positive integer and~$G$ a pro-group, which is assumed abelian if~$n > 1$.
There exists a pointed, connected pro-CW-complex whose~$n\NTH$ homotopy pro-group is isomorphic to~$G$ and whose remaining homotopy pro-groups vanish.
It is unique up to~$\sharp$-isomorphism (i.e. up to morphisms inducing isomorphisms on homotopy pro-groups) and called Eilenberg MacLane space of type~$K(G,n)$.
We say that a connected, locally noetherian scheme~$X$ with geometric point~$\bar{x}$ is $K(\pi,1)$ if the canonical morphism
$$
X_{\et} \to K(\pi_1(X,\bar{x}),1).
$$
is a $\sharp$-isomorphism.

We are interested in a slightly refined version of the $K(\pi,1)$ property:
For a full class of finite groups~$\mc$ and a pro-CW-complex~$Z$ we denote by~$Z(\mc)$ the pro-$\mc$-completion of~$Z$ (which exists by~\cite{AM} Theorem 3.4).
We say that~$X$ is $K(\pi,1)$ with respect to~$\mc$ if $X_{\et}(\mc)$ is $K(\pi,1)$.
Note that in general, being $K(\pi,1)$ neither implies nor is implied by being $K(\pi,1)$ with respect to~$\mc$.
The reason is the following:
For any pro-CW-complex~$Z$ there is a natural isomorphism
$$
\pi_1(Z)(\mc) \overset{\sim}{\to} \pi_1(Z(\mc))
$$
but the higher homotopy pro-groups of~$Z(\mc)$ are not necessarily isomorphic to the~$\mc$-completion of the respective homotopy pro-groups of~$Z$.

There is a criterion for a scheme to be $K(\pi,1)$ with respect to~$\mc$ which involves only \'etale cohomology.
In order to explain it let us fix some terminology:
A Galois~$\mc$-covering of~$X$ is a Galois covering with Galois group in~$\mc$.
A~$\mc$-covering is a covering which is dominated by a Galois~$\mc$-covering.
The \'etale coverings of~$X$ constitute a Galois category by~\cite{SGA1}, Exp.~V, §7.
From this it is easy to deduce that the same holds for the~$\mc$-coverings of~$X$.
We have the following characterization of schemes of type $K(\pi,1)$ (see~\cite{Schmidt07}, Proposition~2.1):

\begin{proposition} \label{kpi1}
 Let~$\mc$ be a full class of finite groups and~$X$ a locally noetherian scheme.
 The following assertions are equivalent:
 \begin{enumerate}[label=(\roman*)]
  \item $X$ is~$K(\pi,1)$ with respect to~$\mc$.
  \item Let~$i \geq 1$ and~$\Lambda = \mZ/\ell\mZ$ with~$\ell \in \mN(\mc)$.
	Then, for every~$\mc$-covering~$X' \to X$ and every class~$\phi \in H^i(X',\Lambda)$ there is a~$\mc$-covering~$X'' \to X'$ such that~$\phi$ maps to zero under
	$$
	H^i(X',\Lambda) \to H^i(X'',\Lambda).
	$$
 \end{enumerate}
\end{proposition}

The reader not familiar with \'etale homotopy theory may safely take this criterion as a definition of the $K(\pi,1)$-property.
Throughout the rest of this article we will work exclusively with the above cohomological characterization.
Note that the condition on the first cohomology group is automatically satisfied as $H^1(X',\Lambda)$ classifies \'etale $\Lambda$-torsors of~$X'$ and these are trivialized by a $\mc$-covering.
 
Let us now consider the situation on arithmetic surfaces:
We fix a base scheme~$B$ which is the spectrum of an excellent Dedekind ring of dimension one.
In this article we are mainly interested in the case where~$B$ is a henselian discrete valuation ring
 but in view of future work on global arithmetic surfaces we formulate most results in more general terms.
It is only in Section\nobreakspace \ref {killingthecohomologyofhigherdirectimages} and in Section\nobreakspace \ref {mainresults} that we restrict our attention to arithmetic surfaces over a henselian base.
By an \emph{arithmetic surface} over $B$ we mean an irreducible, normal scheme~$U$ of dimension~$2$ which is flat and of finite type over~$B$ with geometrically connected generic fiber.
Take a full class of finite groups~$\mc$ such that all prime numbers in~$\mN(\mc)$ are invertible on~$B$.
Proposition\nobreakspace \ref {kpi1} leads us to the question whether a given cohomology class~$\phi \in H^i(U,\Lambda)$ for $i \geq 2$ can be killed by a $\mc$-covering.

Assume there is a compactification~$\bar{X}$ of~$U/B$ and an open subscheme~$X$ of~$\bar{X}$ containing~$U$
 such that $\bar{X}-X$ is the support of a nonempty regular horizontal divisor intersecting all vertical divisors transversally.
Setting $Z = X-U$ (with the reduced scheme structure) we have the following exact sequence:
\begin{equation} \label{excisionsequence}
 \ldots \to H^i_Z(X,\Lambda) \to H^i(X,\Lambda) \to H^i(U,\Lambda) \to H^{i+1}_Z(X,\Lambda) \to \ldots.
\end{equation}
This reduces the task of killing cohomology classes in $H^i(U,\Lambda)$ to killing classes in $H^{i+1}_Z(X,\Lambda)$ and $H^i(X,\Lambda)$.

Let us first have a look at $H^{i+1}_Z(X,\Lambda)$.
Using resolution of singularities we achieve that~$Z$ is a tidy divisor.
This property is slightly stronger than being an snc divisor (see Section\nobreakspace \ref {tidydivisors} for a definition).
For a tidy divisor~$Z$ we can handle $H^{i+1}_Z(X,\Lambda)$ using absolute cohomological purity.
Roughly speaking, a cohomology class in $H^{i+1}_Z(X,\Lambda)$ is killed by a $\mc$-covering which is sufficiently ramified along~$Z$ (see Section\nobreakspace \ref {killingcohomologywithsupport}).

The cohomology groups $H^i(X,\Lambda)$ are accessible because $\pi:X \to B$ is quite close to being an elementary fibration:
There is a base change theorem which asserts that for every geometric point $\bar{b} \to B$ we have
\[
 (R^j \pi_*\Lambda)_{\bar{b}} \cong H^j(X_{\bar{b}},\Lambda)
\]
(see Proposition\nobreakspace \ref {basechange}).
In particular, $R^j \pi_*\Lambda = 0$ for $j \geq 3$.
Moreover, $X_{\bar{b}}$ is an affine curve for all~$\bar{b}$ where~$X_{\bar{b}}$ is regular.
As there are only finitely many singular fibers, this implies that $R^2 \pi_*\Lambda$ is a skyscraper sheaf.

Let us specialize to the case we are primarily interested in in this article, namely where~$B$ is the spectrum of a henselian discrete valuation ring~$R$.
As mentioned before, $H^1(X,\Lambda)$ automatically vanishes in the limit over all $\mc$-coverings.
This leaves us to cope with the second cohomology.
Unfortunately, $H^2(X,\Lambda)$ is not necessarily killed by a $\mc$-covering.
However, a glance at sequence~(\ref{excisionsequence}) reveals that it suffices to show that
\[
 \coker(H^2_Z(X,\Lambda) \to H^2(X,\Lambda))
\]
is killed by a $\mc$-covering.
In Section\nobreakspace \ref {killingthecohomologyofhigherdirectimages} we give conditions for this to be true.

Having carved out conditions for an arithmetic surface to be $K(\pi,1)$ with respect to~$\mc$
 we set out to construct \'etale neighborhoods~$U$ on a given arithmetic surface satisfying these conditions.
The main difficulty lies in ensuring that~$U$ has enough $\mc$-coverings that are sufficiently ramified along the boundary (see Section\nobreakspace \ref {neighborhoodswithenoughtamecoverings}).

This article is based on parts of my thesis written under the supervision of Alexander Schmidt.
I would like to thank him for posing this interesting question and supporting me during the process of answering it.
Moreover, my thanks go to the referee whose suggestions helped me a lot in improving the overall structure of the paper.

\section{Tidy divisors on arithmetic surfaces} \label{tidydivisors}

Let~$X$ be a noetherian scheme.
Throughout this article we identify effective Cartier divisors on~$X$ with the associated closed subschemes of~$X$ whenever this does not lead to confusion.
If the ambient scheme is normal, we can do the same for effective Weil divisors.
Remember that an effective Cartier divisor~$D$ on~$X$ has \emph{simple normal crossings} at a point $x \in D$ if $X$ is regular at~$x$
 and there is a local system of parameters $f_1,\ldots,f_n$ at~$x$ and $m_1,\ldots,m_n \in \mN_0$ such that $f_1^{m_1}\cdot \ldots \cdot f_n^{m_n}$ provides a local equation for~$D$.
The effective Cartier divisor $D$ is a simple normal crossings (snc) divisor if it has simple normal crossings at each point.
Notice that an snc divisor does not necessarily have to be reduced.
We say that two effective Cartier divisors~$D$ and~$D'$ \emph{intersect transversally} at a point $x \in D \cap D'$ if they have no common irreducible component passing through~$x$
 and $D+D'$ has simple normal crossings at~$x$.

Suppose now that $X/B$ is an arithmetic surface.
An effective Cartier divisor~$D$ on~$X$ is \emph{tidy} at a point~$x \in D$
 if for every vertical divisor~$D'$ of~$X$ passing through~$x$, the sum $D+D'$ has simple normal crossings at~$x$.
A \emph{tidy} divisor on~$X$ is an effective Cartier divisor~$D$ which is tidy at every point of~$D$.
In particular, the horizontal irreducible components of a tidy divisor do not intersect.
For a proper closed subscheme~$Z \subseteq X$ we say that a closed point~$z \in Z$ is a \emph{special point} of~$Z$
 if either~$Z$ is not a tidy divisor at~$z$ or~$Z$ is a tidy divisor at~$z$ and~$Z_{\red}$ is singular at~$z$.
The special points of a tidy divisor~$D$ are precisely the points where two irreducible components of~$D$ intersect.
If~$D$ is not tidy but only snc, the special points are the singular points of~$D_{\mathit{red}}$ and the points where~$D$ intersects a vertical divisor whose irreducible components are not contained in the support of~$D$ non-transversally.

Let~$X'$ be another noetherian scheme and $Z \subset X$ and $Z' \subset X'$ closed subschemes.
By a morphism of pairs $f:(X',Z') \to (X,Z)$ we mean a Cartesian square
\[
 \begin{tikzcd}
  Z'	\ar[r]	\ar[d,hookrightarrow]	& Z	\ar[d,hookrightarrow]	\\
  X'	\ar[r]			& X
 \end{tikzcd}
\]
and write~$f$ also for the corresponding morphism $X' \to X$.
We define a \emph{minimal desingularization} of~$(X,Z)$ to be a proper morphism of pairs $\phi:(X',Z') \to (X,Z)$
 such that $X'$ is regular at all points of~$Z'$, the morphism $(X'-Z') \to (X-Z)$ is an isomorphism and~$\phi$ is universal with this property,
 \ie any other proper morphism $\phi'':(X'',Z'') \to (X,Z)$ as above factors through~$\phi$.
Minimal desingularizations of $(X,Z)$ exist by~\cite{MR0491722} and are unique up to unique isomorphism.

\begin{definition}
 Let~$X/B$ be an arithmetic surface and~$Z \subseteq X$ a proper closed subscheme.
 A \emph{tidy desingularization}~$(X',Z') \to (X,Z)$ of~$(X,Z)$ is a birational morphism~$X' \to X$ such that~$Z'$ is a tidy divisor of~$X'$ and~$(X',Z') \to (X,Z)$ factors as
 $$
 (X',Z') = (X_n,Z_n) \to \ldots \to (X_1,Z_1) \to (X_0,Z_0) \to (X,Z),
 $$
 where~$X_0 \to X$ is the minimal desingularization of~$(X,Z)$ and for~$i=1,\ldots,n$ the morphisms~$(X_i,Z_i) \to (X_{i-1},Z_{i-1})$ are blowups of~$X_{i-1}$ in special points of~$Z_{i-1}$.
\end{definition}


\begin{proposition} \label{tidydesingularization}
 Let~$X/B$ be an arithmetic surface and~$Z \subset X$ a proper closed subscheme.
 Then a tidy desingularization of~$(X,Z)$ exists.
\end{proposition}

\begin{proof}
 We may assume that~$X$ is regular at all points of~$X-Z$ as~$X$ is singular in at most a finite set of closed points, which we can remove from~$X$ if they do not lie on~$Z$.
 By~\cite{2009arXiv0905.2191C}, Theorems~0.1 and~0.2 there is a desingularization~$(X',Z') \to (X,Z)$ which is an isomorphism over the complement of~$Z$ such that~$Z'$ is an snc-divisor.
 Moreover, we can assume that~$(X',Z') \to (X,Z)$ is obtained from the minimal desingularization by successive blow-ups in singular, hence special, points.
 Let~$D'$ be the union of~$Z'$ with the finitely many vertical prime divisors containing the points where~$Z'$ intersects a vertical divisor non-transversally.
 After removing from~$X'$ all points of~$D'$ which are not contained in~$Z'$ and where~$D'$ is singular, we may assume that the special points of~$D'$ are contained in~$Z'$.
 By construction, they coincide with the singular points of~$D'_{\red}$.
 Blowing up in singular points of~$D'_{\red}$, we achieve that~$D'$ is an snc-divisor.
 This is equivalent to saying that~$Z'$ is tidy.
 \end{proof}

Let~$\mc$ be a full class of finite groups and denote by $\mN(\mc)$ the submonoid of the positive integers formed by the orders of all groups in~$\mc$.
For an arithmetic surface~$X/B$ such that all elements of $\mN(\mc)$ are invertible on~$X$ and a tidy divisor~$D$ on~$X$
 we want to examine whether $U:=X-D$ is $K(\pi,1)$ with respect to~$\mc$.
The cohomological criterion spelled out in Proposition\nobreakspace \ref {kpi1} leads us to study the $\mc$-coverings of~$U$.
We can extend a $\mc$-covering $U'\to U$ to a finite morphism $X'\to X$ by taking the normalization of~$X$ in~$U'$.
We obtain a $\mc$-covering of $(X,D)$, \ie a finite morphism of pairs $(X_1,D_1) \to (X,D)$ such that $X_1-D_1 \to X-D$ is an \'etale $\mc$-covering.
Any $\mc$-covering of $(X,D)$ is tame,
 \ie the valuations of $K(X)$ associated with the irreducible components of~$D$ are tamely ramified in the corresponding function field extension.
Otherwise there would be an irreducible component~$Z$ of~$D$ whose function field is of characteristic $p > 0$ dividing the degree of the covering.
However, we assumed the orders of all groups in~$\mc$ to be invertible on~$X$, hence not divisible by~$p$.

For a tame covering $(X_1,D_1) \to (X,D)$ the divisor~$D_1$ is not tidy in general.
But we find a tidy desingularization~$(X',D') \to (X_1,D_1)$ using Proposition\nobreakspace \ref {tidydesingularization}.

\begin{definition}
 A \emph{desingularized tame covering} $(X',D') \to (X,D)$ is the composition of a tame covering $(X_1,D_1) \to (X,D)$ and a tidy desingularization $(X',D') \to (X_1,D_1)$.
 In this case we define the exceptional divisor of~$(X',D') \to (X,D)$ to be the exceptional divisor of~$X' \to X_1$.
 A \emph{desingularized $\mc$-covering} is a tame covering $(X',D') \to (X_1,D_1) \to (X,D)$ such that $(X_1,D_1) \to (X,D)$ is a $\mc$-covering.
\end{definition}

\section{Setup and Notation} \label{setupandnotation}

Let~$\mc$ be a full class of finite groups.
Remember that a profinite group~$G$ is called $\mc$-good if for all $G(\mc)$-modules $M \in \mc$ and all $i \in \mN$ the inflation
\[
 H^i(G(\mc),M) \to H^i(G,M)
\]
is an isomorphism.
If~$G$ is the absolute Galois group of a field~$k$, this is equivalent to saying that $\Spec k$ is $K(\pi,1)$ with respect to~$\mc$.
We need a slightly stronger version:
Denote by~$H$ the kernel of the surjection $G \to G(\mc)$.
We say that~$G$ is \emph{strongly $\mc$-good} if for all $G$-modules $M \in \mc$ and all $i \in \mN$
\[
 \alpha_M^i: H^i(G(\mc),M^H) \to H^i(G,M)
\]
is an isomorphism.
This is equivalent to saying that for all $j \ge 1$
\[
 H^j(H,M) = 0.
\]
Indeed, an inspection of the Hochschild-Serre spectral sequence
\[
 H^i(G(\mc),H^j(H,M)) \Rightarrow H^{i+j}(G,M)
\]
shows that~$\alpha_M^i$ is an isomorphism for all $i \ge 0$ if $H^j(H,M) = 0$ for all $j \ge 1$.
For the reverse direction assume that~$\alpha_M^j$ is an isomorphism for all $G$-modules $M \in \mc$.
Then for $j \ge 1$
\begin{IEEEeqnarray*}{rCl}
 H^j(H,M)	& \cong	& H^j(G,\Ind_G^H (M)		\\
		& \cong	& H^j(G(\mc),(\Ind_G^H(M))^H)	\\
		& \cong	& H^j(G(\mc),\Ind_{G(\mc)}(M^H)) = 0.
\end{IEEEeqnarray*}
One example for a strongly $\mc$-good group is~$\hat{\mZ}$, the absolute Galois group of a finite field.
For a prime $p \notin \mN(\mc)$ take an algebraic extension of~$\mQ_p$ containing the $\ell\NTH$ roots of unity for every prime $\ell \in \mN(\mc)$.
Then its absolute Galois group provides another example for a strongly $\mc$-good group.

We now set up the notation for Section\nobreakspace \ref {killingcohomologywithsupport} and~Section\nobreakspace \ref {killingthecohomologyofhigherdirectimages}.
We fix an excellent Dedekind scheme of dimension~$1$.
Furthermore, we take a full class of finite groups~$\mc$ such that all prime numbers~$\ell \in \mN(\mc)$ are invertible on~$B$ and~$\mu_\ell \cong \mZ/\ell\mZ$ on~$B$.
We assume that the absolute Galois groups of the residue fields of~$B$ at closed points are strongly $\mc$-good.

Over~$B$ we fix a proper arithmetic surface~$\bar{X}$ with geometric point~$\bar{x} \to \bar{X}$ lying over a closed point~$x \in \bar{X}$.
Let~$\bar{D}\subseteq \bar{X}$ be a tidy divisor whose support does not contain~$x$.
Let~$\bar{D}_h$ be the maximal subdivisor of~$\bar{D}$ with support on the isolated horizontal components of~$\bar{D}$, \ie, on the horizontal components which do not intersect any other component.
Set~$X = \bar{X} - \bar{D}_h$ and~$U = \bar{X} - \bar{D}$ and denote by~$D \subseteq X$ the restriction of~$\bar{D}$ to~$X$.
We write~$D_v$ for the maximal vertical subdivisor of~$D$ and~$D_h$ for the maximal horizontal subdivisor, such that~$D = D_v + D_h$.
Notice that~$D_v$ is also the maximal vertical subdivisor of~$\bar{D}$.
The maximal horizontal subdivisor of~$\bar{D}$ is given by~$\bar{D}_h + D_h$.
Let~$W$ denote the union of all vertical prime divisors which are contained in a singular fiber of~$\bar{X} \to B$ but not in~$\bar{D}$.
Put differently,~$W$ is the Zariski closure of the union of all reduced fibers~$(U_b)_{\mathit{red}}$ where~$\bar{X}_b$ is singular.
With this notation we also include all singular fibers whose underlying reduced scheme is regular but this does not make much of a difference.
Denote by~$S$ the finite set of special points of~$\bar{D}$, \ie the set of singular points of~$\bar{D}_{\mathit{red}}$.

We denote by~$\mathfrak{I}_{(\bar{X},\bar{D})}$ the category of all pointed desingularized~$\mc$-coverings of~$(\bar{X},\bar{D})$.
We will see in Proposition\nobreakspace \ref {cofiltered} that~$\mathfrak{I}_{(\bar{X},\bar{D})}$ is cofiltered.
Viewing~$\bar{x}$ as geometric point of~$B$ we write~$\mathfrak{I}_B$ for the category of pointed finite \'etale~$\mc$-coverings of~$B$.
Consider the functor $\mathfrak{I}_B \to \mathfrak{I}_{(\bar{X},\bar{D})}$ given by
$$
(B' \to B) \mapsto ((\bar{X} \times_B B',\bar{D} \times_B B') \to (\bar{X},\bar{D})).
$$
It is fully faithful for the following reason:
Take two objects~$B'$ and~$B''$ of~$\mathfrak{I}_B$.
Since the generic fiber of~$\bar{X}$ is geometrically connected, $\bar{X} \times_B B'$ is connected.
It is moreover normal as~$\bar{X}$ is normal and $B' \to B$ is \'etale.
Hence, we can speak of its function field and we have $K(\bar{X} \times_B B') = K(\bar{X}) \otimes_{K(B)} K(B')$.
The same reasoning applies to~$B''$.
In the diagram
\[
 \begin{tikzcd}
  \Hom_{\mathfrak{I}_B}(B'',B')	\ar[r]	\ar[d]	& \Hom_{\mathfrak{I}_{(\bar{X},\bar{D})}}(\bar{X} \times_B B'',\bar{X} \times_B B')	\ar[d]	\\
  \Hom_{K(B)}(K(B'),K(B''))	\ar[r]		& \Hom_{K(B)}(K(B'),K(\bar{X}) \otimes_{K(B)} K(B''))
 \end{tikzcd}
\]
the vertical arrows are injective, the left vertical arrow is even an isomorphism as~$B'$ and~$B''$ are Dedekind schemes.
The lower horizontal map is an isomorphism by the geometric connectedness of the generic fibers of $\bar{X}$.
We conclude that the upper horizontal map is an isomorphism.
Hence, $\mathfrak{I}_B \to \mathfrak{I}_{(\bar{X},\bar{D})}$ is fully faithful and we can view $\mathfrak{I}_B$ as a full subcategory of~$\mathfrak{I}_{(\bar{X},\bar{D})}$.

For $\bar{\pi}:(\bar{X}',\bar{D}') \to (\bar{X},\bar{D})$ in~$\mathfrak{I}_{(\bar{X},\bar{D})}$ let
$$
\bar{X}' \to B' \to B
$$
be the Stein factorization of~$\bar{X}' \to \bar{X} \to B$.
Then~$\bar{X}'$ is an arithmetic surface over~$B'$.
We use analogous notation for~$(\bar{X}',\bar{D}')$ as for~$(\bar{X},\bar{D})$:
We write~$U'$ for~$\bar{X}'-\bar{D}'$ and~$\bar{D}'_h$ for the maximal subdivisor of~$\bar{D}'$ with support on the isolated horizontal components of~$\bar{D}'$ and so on.
Moreover, we write~$E'$ for the exceptional divisor of~$\bar{X}' \to \bar{X}$.

\begin{lemma} \label{preimages}
 Let~$\bar{\pi}:(\bar{X}',\bar{D}') \to (\bar{X}_1,\bar{D}_1) \to (\bar{X},\bar{D})$ be a desingularized $\mc$-covering.
 Then $\bar{\pi}^*\bar{D}_h = \bar{D}'_h$ and~$D'_h$ is the horizontal part of~$\bar{\pi}^*D_h$.
\end{lemma}

\begin{proof}
 Since $\bar{D}' = \bar{\pi}^*\bar{D}$, and by the definition of~$\bar{D}_h$ and~$D_h$, we have
 \[
  \bar{\pi}^*(\bar{D}_h+D_h) = \bar{D}'_h + D'_h + C',
 \]
 where~$C'$ is a vertical divisor with support in~$E'$.
 Moreover, the supports of~$\bar{D}_h$ and~$D_h$ are disjoint.
 Hence, it suffices to show that $\bar{\pi}$ maps the irreducible components of~$\bar{D}'_h$ to $\bar{D}_h$ and those of~$D'_h$ to~$D_h$
  and the part of the statement concerning multiplicities is automatic.
 By the logarithmic version of Abhyankar's lemma (\cite{GR11}, Thm.~7.3.44) $(\bar{X}_1,\bar{D}_1) \to (\bar{X},\bar{D})$ is Kummer \'etale, hence flat.
 Therefore, $\bar{D}_1 \to \bar{D}$ is flat and moreover proper and of finite presentation.
 It is thus open and closed.
 We conclude that a connected component of~$\bar{D}_1$ is mapped surjectively onto a connected component of~$\bar{D}$.
 Furthermore, connected components of~$\bar{D}'$ are mapped surjectively onto connected components of~$\bar{D}_1$.
 Every irreducible component of~$\bar{D}'_h$ is thus mapped to an irreducible component of~$\bar{D}_h$.

 Since~$\bar{D}_h$ is regular and does not intersect the other components of~$\bar{D}$,
  the tame covering~$(\bar{X}_1,\bar{D}_1) \to (\bar{X},\bar{D})$ has regular branch locus in a neighborhood of~$\bar{D}_h$.
 By the generalized Abhyankar lemma (see~\cite{SGA1}, Exp.~XIII, 5.3.0) the preimage~$\bar{D}_{h,1}$ of~$\bar{D}_h$ in~$\bar{X}_1$ is again regular
  and~$\bar{X}_1$ is regular in a neighborhood of~$\bar{D}_{h,1}$.
 In particular,~$\bar{D}_{h,1}$ does not contain special points and thus~$\bar{\pi}^*(\bar{D}_h)$ is contained in~$\bar{D}'_h$.
 Hence, the image of every irreducible component of~$D'_h$ is contained in~$D'_h$.
\end{proof}

As a consequence of Lemma\nobreakspace \ref {preimages} we have $\bar{\pi}^*X = X'$.
We denote the restriction of~$\bar{\pi}$ to~$X'$ by
\[
 \pi: X' \to X.
\]
Notice that $\bar{\pi}^*D_v + C' = D'_v$, where~$C'$ is the vertical divisor defined in the proof of Lemma\nobreakspace \ref {preimages}.
Since the support of~$C'$ is contained in the support of $\bar{\pi}^*D_v$, we obtain
\[
 (\bar{\pi}^*D_v)_{\red} = (D'_v)_{\red}.
\]

Our objective is to investigate whether~$U$ is $K(\pi,1)$ with respect to~$\mc$.
In this article we treat the case where~$B$ is local henselian.
The above setup is more general because we plan another paper on the global case,
 i.e. where~$B$ is an open subscheme of the spectrum of the ring of integers of a finite extension of~$\mQ$.
A great part of the proof is not much more difficult in the global case.
Hence, we prove many propositions in a more general setting and specialize to the local case only when this is considerably easier.

For a morphism of schemes $f:Y \to S$, a closed subscheme~$Z$ of~$Y$, and a sheaf~$\mathcal{F}$ on the \'etale site of~$Y$
 we write $R^j_Z f_* \mathcal{F}$ for the higher direct images with support in~$Z$.
Writing $i:Z \hookrightarrow Y$ for the inclusion, they are the derived functors of $f|_{Z,*}\circ i^!$,
 which sends an \'etale sheaf~$\mathcal{F}$ on~$Y$ to the sheaf
\[
 (S' \to S) \mapsto \ker(\mathcal{F}(Y \times_S S') \to \mathcal{F}((Y-Z) \times_S S')) = H^0_{Z \times_S S'}(Y \times_S S',\mathcal{F})
\]
on the \'etale site of~$S$.
Write $j: V \hookrightarrow Z$ for the open embedding of the complement of~$Z$.
Applying $Rf_*$ to the distinguished triangle
\[
 i_*Ri^!\mathcal{F} \to \mathcal{F} \to Rj_*j^*\mathcal{F} \to i_*Ri^!\mathcal{F}[1]
\]
we obtain a relative excision sequence
\[
 \ldots \to R_Z^if_*\mathcal{F} \to R^if_*\mathcal{F} \to R^if|_{V,*}\mathcal{F}|_V \to \ldots.
\]

\begin{proposition} \label{firstreductions}
 In the above notation assume that~$B$ is a henselian discrete valuation ring
  and that for all primes $\ell \in \mN(\mc)$ setting $\Lambda = \mZ/\ell\mZ$ the following conditions are satisfied:
 \begin{enumerate}[label=(\arabic*)]
  \item $\varinjlim_{\mathfrak{I}_{(\bar{X},\bar{D})}} H^i_{D'}(X',\Lambda) = 0$ for $i \ge 3$ and
  \item $\varinjlim_{\mathfrak{I}_{(\bar{X},\bar{D})}} \coker\big(H^0(B',R^2_{D'} \pi'_* \Lambda) \to H^0(B',R^2 \pi'_* \Lambda)\big) = 0$.
 \end{enumerate}
 Then~$U$ is $K(\pi,1)$ with respect to~$\mc$.
\end{proposition}

\begin{proof}
 By Proposition\nobreakspace \ref {kpi1} it is enough to show that for any $i \ge 2$ and $\Lambda = \mZ/\ell\mZ$ for a prime $\ell \in \mN(\mc)$
\[
 \varinjlim_{U' \to U} H^i(U',\Lambda) = 0,
\]
where the limit is taken over all $\mc$-coverings of~$U$.
Taking the limit of the excision sequences associated with $(X',D')$ for all desingularized $\mc$-coverings $(X',D') \to (X,D)$ we obtain a long exact sequence
\[
 \ldots \to \varinjlim_{\mathfrak{I}_{(\bar{X},\bar{D})}} H^i(X',\Lambda) \to \varinjlim_{\mathfrak{I}_{(\bar{X},\bar{D})}} H^i(U',\Lambda) \to \varinjlim_{\mathfrak{I}_{(\bar{X},\bar{D})}} H^{i+1}_{D'}(X',\Lambda) \to \ldots.
\]
Using condition~(1) we obtain 
\[
 \varinjlim_{\mathfrak{I}_{(\bar{X},\bar{D})}} H^i(U',\Lambda) \cong \varinjlim_{\mathfrak{I}_{(\bar{X},\bar{D})}} H^i(X',\Lambda)
\]
for $i \ge 3$ and an exact sequence
\[
 \varinjlim_{\mathfrak{I}_{(\bar{X},\bar{D})}} H^2_{D'}(X',\Lambda) \to \varinjlim_{\mathfrak{I}_{(\bar{X},\bar{D})}} H^2(X',\Lambda) \to \varinjlim_{\mathfrak{I}_{(\bar{X},\bar{D})}} H^2(U',\Lambda) \to 0.
\]
For a desingularized $\mc$-covering $(X',D') \to (X,D)$ consider the Leray spectral sequence
\[
 H^i(B',R^j\pi'_*\Lambda) \Rightarrow H^{i+j}(X',\Lambda).
\]
Let~$\bar{b}' \to B$ be a geometric point of~$B'$ lying over the closed point~$b'$ of~$B'$ (compatible with~$\bar{x}$).
In Section\nobreakspace \ref {absolutecohomologicalpurity} we will see that absolute cohomological purity and proper base change for $\bar{X}/B$ imply that
 $(R^j\pi'_*\Lambda)_{\bar{b}'} = H^i(X'_{\bar{b}'},\Lambda)$
 (see Proposition\nobreakspace \ref {basechange}).
Since~$X'_{\bar{b}'}$ is a curve over an algebraically closed field, $H^j(X'_{\bar{b}'},\Lambda)$ vanishes for $j \ge 3$.
Moreover, for $i \ge 1$
\[
 \varinjlim_{\mathfrak{I}_{B',\bar{x}}} H^i(B',R^j\pi'_*\Lambda) = \varinjlim_{\mathfrak{I}_{B',\bar{x}}} H^i(k(b'),H^j(X'_{\bar{b}'},\Lambda)) = 0
\]
as the absolute Galois group of $k(b')$ is strongly $\mc$-good by assumption.
In total the above limit vanishes whenever $i+j \ge 3$ and for $(i,j) = (1,1)$ and $(i,j) = (2,0)$.
This implies that
\[
 \varinjlim_{\mathfrak{I}_{(\bar{X},\bar{D})}} H^i(X',\Lambda) = 0
\]
for $i \ge 3$ and
\begin{equation} \label{kerneledge}
 \varinjlim_{\mathfrak{I}_{(\bar{X},\bar{D})}} \ker(H^2(X',\Lambda) \overset{\mathit{edge}}{\longrightarrow} H^0(B',R^2 \pi'_* \Lambda)) = 0.
\end{equation}
Consider the diagram
 \[
  \begin{tikzcd}
   \displaystyle\varinjlim_{\mathfrak{I}_{(\bar{X},\bar{D})}} H^2_{D'}(X',\Lambda)		\ar[r]	\ar[d,"\text{edge}","\sim"']	& \displaystyle\varinjlim_{\mathfrak{I}_{(\bar{X},\bar{D})}} H^2(X',\Lambda)			\ar[r,twoheadrightarrow]	\ar[d,hookrightarrow,"\text{edge}"]	& \displaystyle\varinjlim_{\mathfrak{I}_{(\bar{X},\bar{D})}} H^2(U',\Lambda)	\\
   \displaystyle\varinjlim_{\mathfrak{I}_{(\bar{X},\bar{D})}} H^0(B',R^2_{D'} \pi'_* \Lambda)	\ar[r,twoheadrightarrow]		& \displaystyle\varinjlim_{\mathfrak{I}_{(\bar{X},\bar{D})}} H^0(B',R^2 \pi'_* \Lambda).										&
  \end{tikzcd}
 \]
 The left vertical arrow is an isomorphism because due to purity~$R^j_D \pi_* \Lambda = 0$ for~$j \leq 1$.
 The surjectivity of the lower horizontal arrow is due to condition~(2) and the injectivity of the right vertical arrow is stated above (see~(\ref{kerneledge})).
 We conclude that the upper left horizontal map is surjective, whence
 \[
  \varinjlim_{\mathfrak{I}_{(\bar{X},\bar{D})}} H^2(U',\Lambda) = 0.
 \]
\end{proof}

\section{Absolute cohomological purity} \label{absolutecohomologicalpurity}

Before we go into the details of discussing condition~(1) in Proposition\nobreakspace \ref {firstreductions} we draw some conclusions from Gabber's absolute purity theorem
 that will be crucial in the treatment of cohomology groups with support.

Let~$X$ be a noetherian, regular scheme and $Z \subseteq X$ a regular closed subscheme of pure codimension~$c$.
Then $(X,Z)$ is called a regular pair of codimension~$c$.
Fix a positive integer~$m$ invertible on~$X$ and set~$\Lambda = \mZ/m\mZ$.
The absolute cohomological purity theorem proved by Gabber in~\cite{Fuj00} provides a canonical isomorphism
$$
\underline{H}^n_Z(\Lambda) \cong \begin{cases}
				  0		& \text{for } n \neq 2c	\\
				  \Lambda_Z(-c)	& \text{for } n = 2c.	\\
				 \end{cases}
$$
which is induced from the cycle class map sending~$1 \in \Lambda$ to the fundamental class~$s_{Z/X} \in H^{2c}_Z(X,\Lambda(c))$.
Since the \'etale site of a scheme is equivalent to the \'etale site of its reduction, the statement also holds if only~$X_{\mathit{red}}$ and~$Z_{\mathit{red}}$ are regular.
We call~$(X,Z)$ a \emph{weakly regular pair} if $(X_{\mathit{red}},Z_{\mathit{red}})$ is a regular pair.
Taking into account that the pullback of the fundamental class~$s_{Z_{\mathit{red}}/X_{\mathit{red}}}$ under a morphism $(X',Z') \to (X,Z)$ of weakly regular pairs of codimension~$c$
 is $e \cdot s_{Z'_{\mathit{red}}/X'_{\mathit{red}}}$, where~$e$ denotes the ramification index, we obtain the following compatibility of purity isomorphisms:

\begin{proposition} \label{purityfunctoriality}
 Let~$f:(X',Z') \to (X,Z)$ be a morphism of weakly regular pairs of codimension~$c$.
 Suppose that~$Z$ and~$Z'$ are irreducible and as cycles on~$X'_{\mathit{red}}$ we have~$f_{\mathit{red}}^*Z_{\mathit{red}} = e \cdot Z'_{\mathit{red}}$ with a positive integer~$e$
  (the ramification index).
 Then, for any~$m \in \mN$ invertible on~$X$ the following diagram commutes
 \begin{center}
  \begin{tikzcd}
   H^n_Z(X,\mZ/m\mZ)		\ar[dd]	& H^{n-2c}(Z,\mZ/m\mZ(-c))	\ar[l,"\mathit{purity}","\sim"']		\ar[d]	\\
					& H^{n-2c}(Z',\mZ/m\mZ(-c))	\ar[d,"\cdot e"]					\\
   H^n_{Z'}(X',\mZ/m\mZ)		& H^{n-2c}(Z',\mZ/m\mZ(-c)).	\ar[l,"\mathit{purity}","\sim"']
  \end{tikzcd}
 \end{center}
\end{proposition}

\begin{corollary} \label{zeromap}
 Let~$X$ be a noetherian, regular scheme and~$f:X' \to X$ a tamely ramified covering such that the branch locus~$D \subseteq X$ is regular.
 Let~$Z$ be a regular closed subscheme of~$D$ and let~$Z'$ denote its preimage in~$X'$.
 Then, for any integer~$m$ dividing the ramification index of each irreducible component of~$Z'$, the canonical map
 \begin{equation*}
  H^n_Z(X,\mZ/m\mZ) \to H^n_{Z'}(X',\mZ/m\mZ)
 \end{equation*}
 is the zero map for all~$n \in \mN$.
\end{corollary}

\begin{proof}
 Without loss of generality we may assume that~$Z$ is integral.
 The scheme $X'$ and the underlying reduced subscheme of~$Z'$ are regular because the branch locus~$D$ is regular.
 Denote by~$Z'_k$ for $k=1,\ldots ,r$ the irreducible components of~$Z'$.
 For each~$k$ we can now apply Proposition\nobreakspace \ref {purityfunctoriality} to the morphism
 \begin{equation*}
  X'- \bigcup_{i \neq k} Z'_i \to X
 \end{equation*}
 to conclude that
 \begin{equation*}
  H^n_Z(X,\mZ/m\mZ) \to H^n_{Z'_k}(X'- \bigcup_{i \neq k} Z'_i,\mZ/m\mZ)
 \end{equation*}
 is the zero map.
 But
 \begin{equation*}
  H^n_{Z'}(X',\mZ/m\mZ) = \bigoplus_k H^n_{Z'_k}(X'- \bigcup_{i \neq k} Z'_i,\mZ/m\mZ),
 \end{equation*}
 and the corollary follows.
\end{proof}

Using Gabber's absolute purity theorem we can prove the following refined version of the proper base change theorem.

\begin{proposition} \label{basechange}
 Let~$(X,Z)$ be a weakly regular pair of codimension~$c$ and set~$U = X-Z$.
 Let~$\pi : X \to Y$ be a proper morphism such that~$Z$ intersects~$X_y$ transversally for any closed point~$y$ of~$Y$.
 Set~$\Lambda = \mZ/m\mZ$ for an integer~$m$ prime to the residue characteristics of~$X$.
 Then for any geometric point $\bar{y} \to Y$ and any integer~$d$ the base change morphisms
 \begin{equation*}
  (R^n (\pi_U)_* \Lambda(d))_{\bar{y}} \to H^n(U_{\bar{y}},\Lambda(d))
 \end{equation*}
 are isomorphisms for any~$n \geq 0$.
\end{proposition}

\begin{proof}
 Without loss of generality we may assume~$Y$ is the spectrum of a strictly henselian local ring with closed point~$y$.
 Then,~$\mu_m \cong \mZ/m\mZ$ on~$X$ and it suffices to prove the lemma for~$d=0$.
 We need to show that
 \begin{equation*}
  H^n(U,\Lambda) \to H^n(U_y,\Lambda)
 \end{equation*}
 is an isomorphism.
 Consider the following diagram of excision sequences
 \begin{center}
 \begin{tikzcd}
  \ldots	\ar[r]	& H^n_Z(X,\Lambda)		\ar[d]	\ar[r]	& H^n(X,\Lambda)	\ar[d]	\ar[r]	& H^n(U,\Lambda)	\ar[d]	\ar[r]	& \ldots .	\\
  \ldots	\ar[r]	& H^n_{Z_y}(X_y,\Lambda)		\ar[r]	& H^n(X_y,\Lambda)		\ar[r]	& H^n(U_y,\Lambda)		\ar[r]	& \ldots
 \end{tikzcd}
 \end{center}
The homomorphisms~$H^n(X,\Lambda) \to H^n(X_y,\Lambda)$ are bijective due to proper base change.
Since by assumption~$Z$ intersects~$X_y$ transversally, $(X_y,Z_y) \to (X,Z)$ is a morphism of weakly regular pairs of codimension~$c$ yielding a commutative diagram
 \begin{center}
 \begin{tikzcd}
  H^n_Z(X,\Lambda)		\ar[r,"\sim"]	\ar[d]	& H^{n-2c}(Z,\Lambda(-c))	\ar[d]	\\
  H^n_{Z_y}(X_y,\Lambda)	\ar[r,"\sim"]		& H^{n-2c}(Z_y,\Lambda(-c)).
 \end{tikzcd}
 \end{center}
The horizontal maps are purity isomorphisms and the vertical map on the right is an isomorphism by proper base change.
Hence, the vertical map on the left is an isomorphism and the lemma follows by applying the five lemma to the above diagram of exact sequences.
\end{proof}

Finally, we prove the following technical result which is a variant of (and follows from) the compatibility of the purity isomorphisms
 associated with subschemes $Y \subseteq Z \subseteq X$ such that $(X,Y)$, $(X,Z)$ and $(Z,Y)$ are weakly regular pairs:

\begin{proposition} \label{commutative}
 Let~$X/B$ be an arithmetic surface and~$D \subset X$ an snc-divisor.
 Let~$S \subset D$ be a set of closed points containing the set~$D_{\mathit{sing}}$ of singular points of~$D$.
 Denote by~$D_N$ the normalization of~$D$ and set~$S_N = S \times_D D_N$.
 Then the following diagram of cohomology groups with coefficients in~$\Lambda = \mZ/m\mZ$ ($m$ prime to the residue characteristics of~$X$) commutes
 \begin{center}
  \begin{tikzcd}
   H^3_{D-S}(X-S,\Lambda)	\ar[r]			\ar[rr,bend left=20,"\delta"]					& H^3(X-S,\Lambda)		\ar[r,"\delta"]	& H^4_S(X,\Lambda)						\\
   H^1(D-S,\Lambda(-1))		\ar[u,"purity","\sim"']	\ar[d,dash,shift left=.5]	\ar[d,dash,shift right=.5]	& 						& H^0(S,\Lambda(-2))	\ar[u,"purity","\sim"']			\\
   H^1(D_N-S_N,\Lambda(-1))	\ar[r,"\delta"]										& H^2_{S_N}(D_N,\Lambda(-1))			& H^0(S_N,\Lambda(-2)).	\ar[l,"purity","\sim"']	\ar[u,"norm"]
  \end{tikzcd}
 \end{center}
 All maps~$\delta$ are connecting homomorphisms of excision sequences.
\end{proposition}

\begin{proof}
 Denote by~$D_i$ for $i=1,\ldots r$ the irreducible components of~$D$.
 Since
 \begin{equation*}
  H^3_{D-S}(X-S,\Lambda) = \bigoplus_i H^3_{D_i - S}(X - S,\Lambda),
 \end{equation*}
 it suffices to prove the proposition for each component~$D_i$ separately.
 We may thus assume without loss of generality that~$D$ is a regular irreducible curve.
 In this case the above diagram reduces to
 \begin{center}
  \begin{tikzcd}
   H^3_{D-S}(X-S,\Lambda)	\ar[r]			\ar[rr,bend left=20,"\delta"]	& H^3(X-S,\Lambda)		\ar[r,"\delta"]	& H^4_S(X,\Lambda)							\\
   H^1(D-S,\Lambda(-1))		\ar[u,"purity","\sim"']	\ar[r,"\delta"]			& H^2_{S}(D,\Lambda(-1))			& H^0(S,\Lambda(-2)).	\ar[u,"purity","\sim"']	\ar[l,"purity","\sim"']
  \end{tikzcd}
 \end{center}
Consider the commutative diagram
\begin{center}
 \begin{tikzcd}[column sep=large]
  H^3_{D-S}(X-S,\Lambda)					\ar[r,"\delta"]								& H^4_S(X,\Lambda)									\\
  H^1(D-S,\underline{H}^2_{D-S}(X-S,\Lambda))			\ar[u,"\sim"]				\ar[r,"\delta"]	\ar[d,"\sim"']	& H^2_S(D,\underline{H}^2_D(X,\Lambda))			\ar[u,"\sim"]	\ar[d,"\sim"']	\\
  H^1(D-S,\Lambda(-1)) \otimes H^2_{D-S}(X-S,\Lambda(1))	\ar[r,"\delta \otimes res^{-1}"]					& H^2_S(D,\Lambda(-1)) \otimes H^2_D(X,\Lambda(1))					\\
  H^1(D-S,\Lambda(-1))						\ar[u,"\sim","\otimes s_{D-S/X-S}"']	\ar[r,"\delta"]			& H^2_S(D,\Lambda(-1)).					\ar[u,"\sim","\otimes s_{D/X}"']
 \end{tikzcd}	
\end{center}
 The restriction
 \begin{equation*}
  res : H^2_D(X,\Lambda(1)) \to H^2_{D-S}(X-S,\Lambda(1))
 \end{equation*}
 is an isomorphism which maps the fundamental class~$s_{\nicefrac{D}{X}}$ to~$s_{\nicefrac{D-S}{X-S}}$.
 For this reason, the homomorphism~$\delta \otimes res^{-1}$ in the third line of the diagram is well defined and the lowermost square commutes.
 Commutativity of the middle square follows because~$\underline{H}_D(X)$ is a free sheaf which restricts to~$\underline{H}_{D-S}(X-S)$ on~$D-S$.
 The upper square commutes due to compatibility of the spectral sequences
 \begin{IEEEeqnarray*}{rCl}
  H^i_S(D,\underline{H}^j_D(X,\Lambda))       & \Rightarrow & H^{i+j}_S(X,\Lambda),  \\
  H^i(D-S,\underline{H}^j_{D-S}(X-S,\Lambda)) & \Rightarrow & H^{i+j}_{D-S}(X-S,\Lambda).
 \end{IEEEeqnarray*}
 Furthermore, by~\cite{Fuj00}, Proposition~1.2.1 the following diagram commutes
 \begin{center}
  \begin{tikzcd}[column sep=large]
   H^4_S(X,\Lambda)												& 											\\
   H^2_S(D,\underline{H}^2_D(X,\Lambda))		\ar[u,"\sim"]				\ar[d,"\sim"']	& 											\\
   H^2_S(D,\Lambda(-1)) \otimes H^2_D(X,\Lambda(1))								& 											\\
   H^2_S(D,\Lambda(-1))					\ar[u,"\sim","\otimes s_{D/X}"']			& H^0(S,\Lambda(-2)).	\ar[l,"\sim"',"\text{purity}"]	\ar[uuul,"\sim","\text{purity}"']
  \end{tikzcd}
 \end{center}
Putting the two diagrams together, the assertion of the proposition follows.
\end{proof}

\section{Killing cohomology with support} \label{killingcohomologywithsupport}

In the setup of Section\nobreakspace \ref {setupandnotation} we derive conditions for hypothesis~(1) in Proposition\nobreakspace \ref {firstreductions} to hold.
The idea is to kill cohomology classes with support in~$D$ by $\mc$-coverings of $(X,D)$ which are sufficiently ramified along~$D$ (compare Corollary\nobreakspace \ref {zeromap}).
More precisely we need the following notion:

\begin{definition} \label{enoughtamecoverings}
 Let~$Y$ be an arithmetic surface and $Z \subseteq Y$ an effective Weil divisor.
 We say that~$(Y,Z)$ has \emph{enough tame coverings} at a closed point~$y$ of~$Z$ if for every irreducible component~$C$ of~$Z$ passing through~$y$
  there is~$f \in K(Y)^{\times}$ with support in~$Z$ such that~$\deg_C(f) > 0$ and~$\deg_P(f) = 0$ for any other prime divisor~$P$ passing through~$y$.
 We say that~$(Y,Z)$ has enough tame coverings if it has enough tame coverings at every closed point of~$Z$.
\end{definition}

If~$(Y,Z)$ has enough tame coverings at a point~$y$ and~$C$ is an irreducible component of~$Z$ passing through~$y$,
 we can construct $\mc$-coverings of~$(Y,Z)$ of arbitrarily high ramification index in~$C$ by taking the normalization of~$Y$ in a function field extension $K(Y)(\sqrt[d]{f})|K(Y)$
 with~$f$ chosen as in Definition\nobreakspace \ref {enoughtamecoverings}.
In a neighborhood of~$y$ this covering ramifies only in~$C$.

Unfortunately, in order to treat condition~(1) in Proposition\nobreakspace \ref {firstreductions} we cannot directly apply Corollary\nobreakspace \ref {zeromap} as~$D$ is not necessarily regular.
Instead we proceed in two steps using the excision sequence
\[
 \ldots \to H^i_S(X,\Lambda) \to H^i_D(X,\Lambda) \to H^i_{D-S}(X-S,\Lambda) \to \ldots
\]
and applying Corollary\nobreakspace \ref {zeromap} to $(X-S,D-S)$ and $(X,S)$.
In case $i=3$ the argument is a bit subtle and we need to understand the kernel of the map $H^3_{Z-T}(X-T,\Lambda) \to H^4_T(X,\Lambda)$
 for a vertical divisor $Z \subset X$ and a finite set of closed points~$T$ containing the singular points of~$Z$.
By Proposition\nobreakspace \ref {commutative} the purity isomorphisms translate this map to a map $H^1(Z-T,\Lambda(-1)) \to H^1(S,\Lambda(-2))$ which is defined entirely in terms of curves.
The following lemma describes its kernel.

\begin{lemma} \label{homology}
 Let~$C$ be a projective curve over an algebraically closed field~$k$ with only ordinary double points and let~$\Gamma_C$ denote its dual graph.
 Let~$T \subset C$ be a finite set of closed points containing the set~$C_{\mathit{sing}}$ of singular points of~$C$.
 Define~$C_N := \coprod_i C_i$, where~$C_i$ are the normalizations of the irreducible components of~$C$.
 Set~$T_N= T \times_C C_N$.
 For~$m \in \mN$ prime to the characteristic of~$k$ consider the homomorphism~$\beta$ of cohomology groups with coefficients in~$\mZ/m\mZ$ defined by
 \begin{center}
  \begin{tikzcd}
   H^1(C_N-T_N)	\ar[r,"\alpha"]		& H^2_{T_N}(C_N)	& H^0(T_N)(-1)	\ar[l,"\sim"',"\text{purity}"]	\ar[r,"\text{norm}"']	& H^0(T)(-1),	\\
   H^1(C-T) 	\ar[u,dash,shift left=.5]	\ar[u,dash,shift right=.5]	\ar[urrr, bend right=8,"\beta"]	
  \end{tikzcd}
 \end{center}
 where~$\alpha$ is the connecting homomorphism of the excision sequence associated with~$(C_N,T_N)$.
 Then
 \begin{equation*}
  \frac{\ker(\beta)}{\ker(\alpha)} \cong H_1(\Gamma_C,\mZ/m\mZ),
 \end{equation*}
 where~$H_1(\Gamma_C,\mZ/m\mZ)$ denotes singular homology with coefficients in~$\mZ/m\mZ$.
\end{lemma}

\begin{proof}
 The group~$H_1(\Gamma_C,\mZ/m\mZ)$ can be calculated using a cellular chain complex.
 The zero-skeleton~$(\Gamma_C)_0$ consists of the nodes of the graph which correspond to the irreducible components~$C_i$ and the one-skeleton~$(\Gamma_C)_1$ is all of~$\Gamma_C$.
 Thus, the one-cells are the edges of the graph, which correspond to the points in~$C_{\mathit{sing}}$.
 We give each edge~$s$ a direction by choosing an initial node~$C_1(s)$ and an end node~$C_2(s)$.
 Then~$H_1(\Gamma_C,\mZ/m\mZ)$ is the first homology of the sequence
 \begin{equation*}
  0 \to H_1((\Gamma_C)_1,(\Gamma_C)_0,\mZ/m\mZ) \overset{d}{\to} H_0((\Gamma_C)_0,\mZ/m\mZ) \to 0
 \end{equation*}
 and the map~$d$ can be identified with
 \begin{IEEEeqnarray*}{rCl}
  \bigoplus_{s \in C_{\mathit{sing}}} \mZ/m\mZ \cdot s & \to & \bigoplus_i \mZ/m\mZ \cdot C_i.  \\
  s                                           & \mapsto     & C_2(s) - C_1(s).
 \end{IEEEeqnarray*}
 Let us now compute~$\ker(\beta)/\ker(\alpha)$.
 \begin{IEEEeqnarray*}{rCl}
  \frac{\ker(\beta)}{\ker(\alpha)} & = & \ker\left(\frac{H^1(C-T)}{\ker(\alpha)} \to H^0(T)(-1)\right)  \\
                                 & = & \ker(Im(\alpha) \to H^0(T)(-1))  \\
                                 & = & \ker(\ker(H^2_{T_N}(C_N) \to H^2(C_N)) \to H^0(T)(-1))  \\
                                 & = & \ker(H^2_{T_N}(C_N) \to H^2(C_N)) \cap \ker(H^2_{T_N}(C_N) \to H^0(T)(-1)).
 \end{IEEEeqnarray*}
 We identify the map~$H^2_{T_N}(C_N) \to H^2(C_N)$ with
 \begin{IEEEeqnarray*}{rCl}
  \bigoplus_{s_N\in T_N} \mZ/m\mZ \cdot s_N & \to & \bigoplus_i \mZ/m\mZ \cdot C_i,  \\
  s_N                                     & \mapsto     & C(s_N)
 \end{IEEEeqnarray*}
 where~$C(s_N)$ is the component of~$C_N$ which contains~$s_N$ and $H^2_{T_N}(C_N) \to H^0(T)(-1)$ with
 \begin{IEEEeqnarray}{rCl} \label{normmap}
  \bigoplus_{s_N\in T_N} \mZ/m\mZ \cdot s_N & \to & \bigoplus_{s\in T} \mZ/m\mZ \cdot s,  \\
  s_N                                     & \to & s(s_N)
 \end{IEEEeqnarray}
 where~$s(s_N)$ is the image of~$s_N$ in~$T$.
 In particular, we obtain an isomorphism
 \begin{IEEEeqnarray*}{rCl}
  \bigoplus_{s\in C_{\mathit{sing}}} \mZ/n\mZ \cdot s & \to & \ker(\bigoplus_{s_N\in T_N} \mZ/n\mZ \cdot s_N \to \bigoplus_{s\in T} \mZ/n\mZ \cdot s), \\
  s                                           & \mapsto     & (s_N)_2(s) - (s_N)_1(s)
 \end{IEEEeqnarray*}
 where~$(s_N)_i(s) \in C_i(s)$ are the two preimages of~$s$ in $C_N$.
 Therefore,~$\ker(\beta)/\ker(\alpha)$ is isomorphic to the kernel of the composition
 \begin{equation*}
  \bigoplus_{s\in C_{\mathit{sing}}} \mZ/m\mZ \cdot s \to \bigoplus_{s_N\in T_N} \mZ/m\mZ \cdot s_N \to \bigoplus_i \mZ/m\mZ \cdot C_i,
 \end{equation*}
 which maps~$s\in C_{\mathit{sing}}$ to~$C_2(s) - C_1(s)$.
 Comparing with the calculation of $H_1(\Gamma_C,\mZ/m\mZ)$ at the beginning of the proof we see that
 \begin{equation*}
  \frac{\ker(\beta)}{\ker(\alpha)} \cong H_1(\Gamma_C,\mZ/m\mZ).
 \end{equation*}
\end{proof}

In the following we call a (not necessarily integral) projective curve~$C$ over a field~$k$ a rational tree if $H^1(C,\mO_C) = 0$.
By flat base change~$C$ is rational if and only if its base change~$\bar{C}$ to the separable closure~$\bar{k}$ of~$k$ is rational.
This is the case precisely if every irreducible component of~$\bar{C}$ is isomorphic to $\mP^1_{\bar{k}}$ and moreover the dual graph of~$\bar{C}$ is a tree, \ie simply connected (see~\cite{Deb}, Definition~4.23).

\begin{lemma} \label{rationalsupport}
 Let~$Z \leq D_v$ be a subdivisor whose connected components are rational trees.
  Suppose that for every geometric point~$\bar{b}$ above a closed point~$b \in B$ the natural map
 \begin{equation*}
  \pi_1(b,\bar{b})(\mc) \to \pi_1(B,\bar{b})(\mc)
 \end{equation*}
 is injective.
 Then
 \begin{equation*}
  \varinjlim_{\mathfrak{I}_B} H^3_{S'}(X',\Lambda) \to \varinjlim_{\mathfrak{I}_B} H^3_{Z'\cup S'}(X',\Lambda)
 \end{equation*}
 is surjective
 (Remember that~$S'$ is the set of special points of~$\bar{D}'$.)
\end{lemma}

\begin{proof}
 Using the excision sequence for $S' \subseteq Z' \cup S' \subseteq X'$ we see that the required surjectivity is equivalent to the injectivity of
 \[
  \varinjlim_{\mathfrak{I}_B} H^3_{Z'- S'}(X'-S',\Lambda) \to \varinjlim_{\mathfrak{I}_B} H^4_{S'}(X',\Lambda).
 \]
 In other words, for $B' \to B$ in $\mathfrak{I}_B$ and $\varphi \in H^3_{Z'- S'}(X'-S',\Lambda)$
  we have to construct $B'' \to B'$ in $\mathfrak{I}_B$ such that~$\varphi$ maps to zero in $H^4_{S''}(X'',\Lambda)$.
 As the assumptions are stable under \'etale base change, we may assume $B' = B$.
  By Proposition\nobreakspace \ref {commutative} we have the following commutative diagram
 \begin{center}
  \begin{tikzcd}
  H^3_{Z-S}(X-S,\Lambda)	\ar[r]					& H^4_{S}(X,\Lambda)	\ar[d]	\\
  H^1(Z-S,\Lambda(-1))	\ar[r,"\beta(-1)"]	\ar[u,"\sim"]	& H^0(S,\Lambda(-2)),
  \end{tikzcd}
 \end{center}
 where~$\beta(-1)$ is the~$(-1)$-twist of the map~$\beta$ defined in Lemma\nobreakspace \ref {homology}.
 It thus suffices to show that the kernel of~$\beta$ vanishes in the limit over~$\mathfrak{I}_B$.
 Without loss of generality we may assume that~$Z$ is connected.
 In particular, it is contained in a single closed fiber of~$X \to B$ over some point~$b \in B$ with residue field~$k(b)$.
 Let~$\overline{k(b)}$ be an algebraic closure of~$k(b)$ and denote by~$\bar{Z}$ and~$\bar{S}$ the base change of~$Z$ and~$S$, respectively, to~$\overline{k(b)}$.
 Moreover, write~$Z_N$ for the normalization of~$Z$ and~$\bar{Z}_N$ for its base change to~$\overline{k(b)}$.
 Consider the diagram of cohomology groups with coefficients in~$\Lambda$
 \begin{center}
 \begin{tikzcd}
		& H^2(k(b))^d			\ar[r,"="]		& H^2(k(b))^d								&				\\
  0	\ar[r]	& H^1(\bar{Z}_N)^{G_{k(b)}}	\ar[r]	\ar[u]		& H^1(\bar{Z}-\bar{S})^{G_{k(b)}}	\ar[r,"\bar{\beta}"]	\ar[u]	& H^0(\bar{S})(-1)^{G_{k(b)}}	\\
  0	\ar[r]	& H^1(Z_N)			\ar[r]		\ar[u]	& H^1(Z-S)				\ar[r,"\beta"]		\ar[u]	& H^0(S)(-1),			\ar[u]	\\
		& H^1(k(b))^d			\ar[r,"="]	\ar[u]	& H^1(k(b))^d							\ar[u]	& 						\\
		& 0						\ar[u]	& 0								\ar[u]	&
 \end{tikzcd}
 \end{center}
 where~$d$ is the number of components of~$Z_N$.
 The vertical sequences are induced by the Hochschild-Serre spectral sequences
 \begin{IEEEeqnarray*}{rCl}
  H^i(k(b),H^j(\bar{Z}_N),\Lambda))			&\Rightarrow& H^{i+j}(Z_N,\Lambda),	\\
  H^i(k(b),H^j(\bar{Z}_N-\bar{S}_N,\Lambda))		&\Rightarrow& H^{i+j}(Z_N-S_N,\Lambda).
 \end{IEEEeqnarray*}
 The upper horizontal sequence is exact by the following reason:
 According to Lemma\nobreakspace \ref {homology}, the first homology group~$H_1(\Gamma_Z,\Lambda)$ of the dual graph~$\Gamma_Z$ of~$\bar{Z}$
  is isomorphic to~$\ker(\bar{\beta})/\ker(\bar{\alpha})$, where~$\bar{\alpha}$ denotes the connecting homomorphism of the excision sequence associated with~$\bar{S}_N \hookrightarrow \bar{Z}_N$.
 As~$Z$ is a rational tree,~$\Gamma_Z$ is simply connected, and thus its first homology group vanishes.
 It follows that the kernel of~$\bar{\beta}$ equals the image of the map
 \begin{equation*}
  \gamma :H^1(\bar{Z}_N,\Lambda) \hookrightarrow H^1(\bar{Z}_N-\bar{S}_N,\Lambda) = H^1(\bar{Z}-\bar{S},\Lambda).
 \end{equation*}
 Taking~$G_{k(b)}$-invariants we obtain the upper sequence of the above diagram, which is therefore exact.
 A diagram chase now shows the exactness of the lower horizontal sequence.

 Again by the rationality assumption on~$Z$, the cohomology group~$H^1(\bar{Z}_N)$ vanishes.
 The above diagram shows that the kernel of~$\beta$ equals~$H^1(k(b))^d$.
 By the assumption on fundamental groups it vanishes in the limit over~$\mathfrak{I}_B$.
 \end{proof}

\begin{proposition} \label{etalecovering}
 Suppose that the following conditions are satisfied:
 \begin{enumerate}[label=(\roman*)]
  \item $(\bar{X},\bar{D})$ has enough tame coverings.
  \item Every connected component of~$D$ has at least one horizontal component.
  \item For every geometric point~$\bar{b}$ above a closed point~$b \in B$ the natural map
        \begin{equation*}
         \pi_1(b,\bar{b})(\mc) \to \pi_1(B,\bar{b})(\mc)
        \end{equation*}
	is injective.
 \end{enumerate}
 Then, for any $n \geq 3$
 $$
 \varinjlim_{\mathfrak{I}_{(\bar{X},\bar{D})}} H^n_{D'}(X',\Lambda) = 0.
 $$
\end{proposition}

\begin{proof}
 Let~$(\bar{X}',\bar{D}')$ be an object of $\mathfrak{I}_{(\bar{X},\bar{D})}$ and~$\varphi$ an element of $H^n_{D'}(X',\Lambda)$.
 We have to show that there is a desingularized $\mc$-covering $(\bar{X}'',\bar{D}'') \to (\bar{X}',\bar{D}')$
  such that the image of~$\varphi$ in $H^n_{D''}(X'',\Lambda)$ is zero.
 We will see in Proposition\nobreakspace \ref {enoughtamecoveringsstable} that the property of having enough tame coverings is stable under desingularized tame coverings.
 Moreover, it is easy to check that this is true for the remaining assumptions as well.
 Hence, we may assume $(\bar{X}',\bar{D}') = (\bar{X},\bar{D})$.
 We first construct $(\bar{X}',\bar{D}')$ in $\mathfrak{I}_{(\bar{X},\bar{D})}$ such that the image of~$\varphi$ in $H^n_{D'}(X',\Lambda)$ lifts to $H^n_{S'}(X',\Lambda)$.
 
 Let us treat the case~$n=3$.
 Since~$(\bar{X},\bar{D})$ has enough tame coverings, there is a desingularized~$\mc$-covering
 $$
 (X',D') \to (X_1,D_1) \to (X,D),
 $$
 such that~$m$ divides the ramification index of each irreducible component of~$D_1$.
 We have the following commutative diagram of excision sequences with coefficients~$\Lambda$:
 \begin{center}
  \begin{tikzcd}
   H^3_S(X)		\ar[r]	\ar[d]	& H^3_D(X)	\ar[r]	\ar[d]	& H^3_{D-S}(X-S)			\ar[d]	\\
   H^3_{S'\cup E'}(X')	\ar[r]		& H^3_{D'}(X')	\ar[r]		& H^3_{D'-S'\cup E'}(X'-S'\cup E').
  \end{tikzcd}
 \end{center}
 Let~$\varphi'$ denote the image of~$\varphi$ in~$H^3_{D'}(X',\Lambda)$.
 Applying Corollary\nobreakspace \ref {zeromap} to $X' - S'\cup E' \to X - S$, we conclude that the rightmost vertical map is the zero map.
 Consequently,~$\varphi'$ is mapped to zero in~$H^3_{D'-S'\cup E'}(X'-S'\cup E',\Lambda)$.
 Hence, there is~$\varphi'_1 \in H^3_{S'\cup E'}(X',\Lambda)$ mapping to~$\varphi'$.
 In Proposition\nobreakspace \ref {simplyconnected} we will see that the exceptional fibers of a desingularized $\mc$-covering are always rational trees.
 Therefore, we can apply Lemma\nobreakspace \ref {rationalsupport} with~$Z=E'$ to obtain a finite \'etale~$\mc$-covering~$B'' \to B'$
  and thus via base change a finite \'etale~$\mc$-covering~$\bar{X}'' \to \bar{X}'$
   such that the image of~$\varphi'_1$ in~$H^3_{S''\cup E''}(X'',\Lambda)$ lifts to an element $\varphi''_2 \in H^3_{S''}(X'',\Lambda)$.
 Changing notation we may assume that~$\varphi$ lifts to~$\varphi_2 \in H^3_S(X,\Lambda)$.

 Now assume that~$n \geq 4$.
 By the same argument as for~$n=3$ there is a desingularized~$\mc$-covering~$(\bar{X}',\bar{D}') \to (\bar{X},\bar{D})$
  such that the image of~$\varphi$ in~$H^n_{D'}(X',\Lambda)$ lifts to~$H^n_{S'\cup E'}(X',\Lambda)$.
 In particular, it lifts to~$H^n_{D'_v}(X',\Lambda)$ and thus we may assume that~$\varphi$ lifts to~$H^n_{D_v}(X,\Lambda)$ right away.
 
 Consider the excision sequence
 $$
 \ldots \to H^n_S(X,\Lambda) \to H^n_{D_v}(X,\Lambda) \to H^n_{D_v-S}(X-S,\Lambda) \to \ldots.
 $$
 By purity we have
 \begin{equation*}
  H^n_{D_v-S}(X-S,\Lambda) \cong H^{n-2}(D_v-S,\Lambda(-1)).
 \end{equation*}
 For each component~$Z_i$ of~$D_v$ lying over a closed point~$b_i \in B$ with geometric point~$\bar{b}_i$ consider the Hochschild-Serre spectral sequence
 \begin{equation*}
  H^r(b_i,H^s((Z_i - S)_{\bar{b}_i},\Lambda)) \Rightarrow H^{r+s}(Z_i - S,\Lambda).
 \end{equation*}
 Since $(Z_i - S)_{\bar{b}_i}$ is an affine curve over an algebraically closed field, its cohomological dimension is less or equal to one.
 Moreover, for $r \ge 1$, $H^r(b_i,H^s((Z_i - S)_{\bar{b}_i},\Lambda))$ vanishes
  in the limit over~$\mathfrak{I}_B$ as the absolute Galois group of $k(b_i)$ is $\mc$-good and $\pi_1(b,\bar{b})(\mc) \to \pi_1(B,\bar{b})(\mc)$ is injective.
 We conclude that $H^{n-2}(Z_i - S,\Lambda)$ vanishes in the limit over~$\mathfrak{I}_B$ for~$n \geq 4$.
 As before we replace~$\bar{X}$ by~$\bar{X}'$ and may assume that~$\varphi_1$ maps to~$0$ in~$H^n_{D_v-S}(X-S,\Lambda)$.
 Hence,~$\varphi_1$ lifts to~$\varphi_2 \in H^n_S(X,\Lambda)$.
 
 Having lifted~$\varphi$ to~$\varphi_2 \in H^n_S(X,\Lambda)$ for any $n \geq 2$ we now construct $(\bar{X}',\bar{D}')$ in~$\mathfrak{I}_{(\bar{X},\bar{D})}$
  such that~$\varphi_2$ maps to zero in $H^n_{D'}(X',\Lambda)$.
 The cohomology group $H^n_S(X,\Lambda)$ is the direct sum of all $H^n_s(X,\Lambda)$ for the finitely many points $s \in S$.
 For~$s \in S$ choose an irreducible component~$D_s$ of~$D$ passing through~$s$.
 Since~$(\bar{X},\bar{D})$ has enough tame coverings, we find a desingularized~$\mc$-covering $(\bar{X}',\bar{D}') \to (\bar{X}_1,\bar{D}_1) \to (\bar{X},\bar{D})$
  such that~$m$ divides the ramification indices of all irreducible components of~$\bar{D}_1$ lying over~$D_s$ and is unramified in all other prime divisors passing through~$s$.
 Since the branch locus is regular in a neighborhood of~$s$, the pair~$(\bar{X}_1,\bar{D}_1)$ is regular at all preimage points~$s'_1,\ldots,s'_r$ of~$s$.
 Hence, $\bar{X}' \to \bar{X}_1$ is an isomorphism in a neighborhood of~$s'_1,\ldots,s'_r$.
 Therefore, by Corollary\nobreakspace \ref {zeromap}, the homomorphism
 $$
 H^n_s(X,\Lambda) \to \bigoplus_i H^n_{s'_i}(X',\Lambda)
 $$
 is the zero map.
 Take a desingularized~$\mc$-covering~$(\bar{X}'',\bar{D}'') \to (\bar{X},\bar{D})$ dominating the coverings~$(\bar{X}',\bar{D}') \to (\bar{X},\bar{D})$ constructed for each~$s \in S$.

 We obtain a commutative diagram
  \begin{center}
   \begin{tikzcd}
    H^n_S(X,\Lambda)			\ar[r]	\ar[d]	& H^n_D(X,\Lambda)		\ar[d]	\\
    H^n_{S''\cup E''}(X'',\Lambda)	\ar[r]		& H^n_{D''}(X'',\Lambda),	
   \end{tikzcd}
 \end{center}
 where the left vertical homomorphism is the zero map.
 This implies the assertion.
\end{proof}

\section{Killing the Cohomology of higher direct images} \label{killingthecohomologyofhigherdirectimages}

In this section we still keep the notation of Section\nobreakspace \ref {setupandnotation} and examine condition~(2) of Proposition\nobreakspace \ref {firstreductions}, i.e. we strive to kill the cokernel of
\[
 H^0(B,R^2_D \pi_* \Lambda) \to H^0(B,R^2 \pi_* \Lambda)).
\]
In the following lemma we explain how to relate this homomorphism with the intersection matrix of the irreducible components of the singular fibers.

\begin{lemma} \label{intersectionmatrix}
 Suppose that~$B$ is strictly henselian with closed point~$s$.
 Denote by~$\rho$ the intersection matrix of the components of the special fiber of~$\bar{\pi} : \bar{X} \to B$.
 Then, for any integer~$c$ the following diagram commutes
 \begin{center}
  \begin{tikzcd}
   H^2_{D_v}(X,\Lambda(c+1))					\ar[r]					& H^2(X,\Lambda(c+1))	\ar[d,"\text{base change}","\sim"']						\\
   H^0(D_v,\Lambda(c))						\ar[u,"purity","\sim"']			& H^2(X_s,\Lambda(c+1))	\ar[d,"\text{deg}","\sim"']							\\
   \displaystyle \bigoplus_{C \subseteq D_v} \Lambda(c) \cdot C	\ar[u,"\sim"']		\ar[r,"\rho"]	& \displaystyle \bigoplus_{C \subseteq \bar{X}_s, C \cap \bar{D}_h = \emptyset} \Lambda(c) \cdot C.
  \end{tikzcd}
 \end{center}
\end{lemma}

\begin{proof}
 It suffices to prove the lemma for~$c = 0$.
 Consider the commutative diagram
 \begin{center}
  \begin{tikzcd}
   H^2_{D_v}(X,\mu_m)								\ar[r]				& H^2(X,\mu_m)			\ar[r,"\sim"]			& H^2(X_s,\mu_m)	\\\
   \displaystyle \bigoplus_{C \subseteq D_v} H^1_{C}(X,\mG_m) \otimes \Lambda	\ar[r]		\ar[u,"\sim"]	& Pic(X) \otimes \Lambda	\ar[r]		\ar[u,"\sim"]	& \displaystyle \bigoplus_{C \cap \bar{D}_h = \emptyset} Pic(C) \otimes \Lambda		\ar[u,"\sim"]	\ar[d,"\sim"',"(\mathit{deg}_C)_C"]	\\
   \displaystyle \bigoplus_{C \subseteq D_v} \Lambda \cdot C			\ar[u,"\sim"]	\ar[rr]		& 								& \displaystyle \bigoplus_{C \cap \bar{D}_h = \emptyset} \Lambda \cdot C.
  \end{tikzcd}
 \end{center}
 The direct sums on the right hand side run only over irreducible components of~$\bar{X}_s$ with trivial intersection with~$\bar{D}_h$
  as these are precisely the components of~$X_s$ which are proper over~$B$.
 The upper right horizontal isomorphism comes from Proposition\nobreakspace \ref {basechange}.
 The upper vertical maps are connecting homomorphisms of the Kummer sequence.
 The concatenation of the left hand side vertical arrows yields the purity isomorphism and the right hand vertical arrows give the degree map on~$H^2(X_s,\mu_m)$.
 The restrictions
 \begin{equation*}
  Pic(X) \to Pic(C)
 \end{equation*}
 are given by~$D \mapsto D \cdot C$ where~$D \cdot C$ denotes the intersection product of the divisor~$D$ with the curve~$C$.
 Composition with~$\deg_C$ yields the intersection number~$(D \cdot C)$.
 We conclude that the lower horizontal map is indeed given by the intersection matrix~$\rho_{C_1,C_2} = (C_1 \cdot C_2)$.
\end{proof}

We set
$$
\mZ(\mc) = \varprojlim_{n \in \mN(\mc)} \mZ/n\mZ = \prod_{\ell \in \mN(\mc)~\text{prime}} \mZ_{\ell}.
$$

\begin{lemma} \label{imagedivisiblebym}
 Assume that~$(\bar{X},\bar{D})$ has enough tame coverings.
 Then, for every integer~$d \in \mN(\mc)$ there is a desingularized~$\mc$-covering~$(\bar{X}',\bar{D}') \to (\bar{X},\bar{D})$ such that the image of
 \begin{equation*}
  H^2_D(X,\mZ(\mc)(1)) \to H^2_{D'}(X',\mZ(\mc)(1))
 \end{equation*}
 is divisible by~$d$.
\end{lemma}

\begin{proof}
 By purity we have
 $$
 \bigoplus_{C \subseteq D} \mZ(\mc) \cdot C \overset{\sim}{\to} H^2_{D}(X,\mZ(\mc)(1)).
 $$
 Moreover, if~$(\bar{X}',\bar{D}') \to (\bar{X},\bar{D})$ is a desingularized~$\mc$-covering, the induced map
 $$
 \bigoplus_{C \subseteq D} \mZ(\mc) \cdot C \to \bigoplus_{C' \subseteq D'} \mZ(\mc) \cdot C'
 $$
 is given by the pull-back of divisors.
 Let~$C \subseteq D$ be an irreducible component and~$c \in C$ a closed point of~$D$.
 Since~$(\bar{X},\bar{D})$ has enough tame coverings, there is~$f_c \in K(\bar{X})^{\times}$ with support in~$\bar{D}$
  such that in a Zariski neighborhood~$U_c$ of~$c$ we have~$div~f_c = m_c C$ with $m_c > 0$.
 Denote by~$m'_c$ the maximal factor of~$m_c$ contained in~$\mN(\mc)$.
 Let~$\phi_c:(\bar{X}_c,\bar{D}_c) \to (\bar{X},\bar{D})$ be a desingularized~$\mc$-covering with function field extension
 $$
 K(\bar{X}_c) = K(\bar{X})\big(\sqrt[m'_cd]{f_c}\big)|K(\bar{X}).
 $$
 Then~$div~f_c$ is divisible by~$m'_c d$ as an element of $\text{Div}~X_c$.
 Thus,~$\phi_c^*|_{U_c}(C|_{U_c})$ is divisible by~$d$, \ie, the coefficients of all irreducible components of~$\phi_c^*(C)$ whose generic points lie over~$U_c$ are divisible by~$d$.
 This property is conserved by further desingularized coverings.

 There are finitely many closed points~$c_1,\ldots,c_n \in C$ such that the open subschemes~$U_{c_1},\ldots,U_{c_n}$ cover~$C$.
 Let~$(\bar{X}',\bar{D}') \to (\bar{X},\bar{D})$ be a desingularized~$\mc$-covering dominating all coverings $(\bar{X}_{c_i},\bar{D}_{c_i}) \to (\bar{X},\bar{D})$ constructed above.
 Then the pullback of~$C$ to~$\bar{X}'$ is divisible by~$d$.
\end{proof}

\begin{lemma} \label{cokernelvanishes}
 Let~$B_0$ be the strict henselization of~$B$ at a geometric point of~$B$ over a closed point.
 Denote by~$X_0$ and~$D_0$ the base change of~$X$ and~$D$, respectively, to~$B_0$.
 Assume that~$\bar{D}_h$ is nonempty and meets all irreducible components of~$W$.
 If~$(\bar{X},\bar{D})$ has enough tame coverings, the cokernel of
 \begin{equation*}
  H^2_{D_0}(X_0,\mZ(\mc)(1)) \to H^2(X_0,\mZ(\mc)(1))
 \end{equation*}
 vanishes in the limit over~$\mathfrak{I}_{(\bar{X},\bar{D})}$.
\end{lemma}

\begin{proof}
 We may replace~$B$ by~$B_0$.
 We just have to check that all tame coverings of~$(\bar{X}_0,\bar{D}_0)$ occurring in the proof come from coverings of $(\bar{X},\bar{D})$.
 It suffices to prove that the cokernel of
 $$
 \phi: H^2_{D_v}(X,\mZ(\mc)(1)) \to H^2(X,\mZ(\mc)(1))
 $$
 vanishes in the limit over~$\mathfrak{I}_{(\bar{X},\bar{D})}$ as~$H^2_{D_v}(X,\mZ(\mc)(1))$ is a direct summand of $H^2_D(X,\mZ(\mc)(1))$.
 Taking the inverse limit over all~$\Lambda \cong \mZ/n\mZ$ with~$n \in \mN(\mc)$ of the diagrams in Lemma\nobreakspace \ref {intersectionmatrix} and setting $c=0$, we obtain
 \begin{center}
  \begin{tikzcd}
   H^2_{D_v}(X,\mZ(\mc)(1))					\ar[r,"\phi"]					& H^2(X,\mZ(\mc)(1))							\ar[d,"\text{base change}","\sim"']	\\
   H^0(D_v,\mZ(\mc))						\ar[u,"\text{purity}","\sim"']			& H^2(X_s,\mZ(\mc)(1))						\ar[d,"\text{deg}","\sim"']		\\
   \displaystyle \bigoplus_{C \subseteq D_v} \mZ(\mc) \cdot C	\ar[u,"\sim"']			\ar[r,"\rho"]	& \displaystyle \bigoplus_{C \subseteq D_v} \mZ(\mc) \cdot C.
  \end{tikzcd}
 \end{center}
 Since we assumed that~$\bar{D}_h$ meets all irreducible components of~$W$, we have that $C \cap \bar{D}_h = \emptyset$ if and only if~$C \subseteq D_v$.
 By~\cite{Liu}, Theorem~9.1.23 the intersection matrix of the components of the special fiber is negative semidefinite and its radical is generated by a rational multiple of the special fiber.
 Since we assumed that~$\bar{D}_h$ is nonempty, the support of~$D$ does not comprise all irreducible components of the special fiber.
 Hence, the restriction~$\rho$ of the intersection matrix to the components of~$D_v$ is negative definite.
 We conclude that
 \begin{equation*}
  \phi \otimes \mQ : H^2_D(X,\mZ(\mc)(1)) \otimes \mQ  \to H^2(X,\mZ(\mc)(1)) \otimes \mQ
 \end{equation*}
 is an isomorphism and thus the cokernel of~$\phi$ is~$\mc$-torsion.
 Take~$d \in \mN(\mc)$ such that~$d \cdot \coker \phi = 0$.
 By Lemma\nobreakspace \ref {imagedivisiblebym} there is a desingularized~$\mc$-covering~$(\bar{X}',\bar{D}') \to (\bar{X},\bar{D})$ of~$(\bar{X},\bar{D})$ such that the image of
 \begin{equation*}
  H^2_D(X,\mZ(\mc)(1)) \to H^2_{D'}(X',\mZ(\mc)(1))
 \end{equation*}
 is divisible by~$d$.
 Taking into account that on $H^2(X',\mZ(\mc)(1))$ multiplication by~$d$ is injective, this proves the result.
\end{proof}

We can now specify sufficient conditions for assertion~(2) in Proposition\nobreakspace \ref {firstreductions} to hold:

\begin{proposition} \label{lerayfiltration}
 Assume that~$\bar{D}_h$ is nonempty and intersects all irreducible components of~$W$ and that~$(\bar{X},\bar{D})$ has enough tame coverings.
 Then 
 \[
  \coker (\varinjlim_{\mathfrak{I}_{(\bar{X},\bar{D})}} H^0(B',R^2_{D'} \pi'_* \Lambda) \to \varinjlim_{\mathfrak{I}_{(\bar{X},\bar{D})}} H^0(B',R^2 \pi'_* \Lambda)) = 0.
 \]
\end{proposition}

\begin{proof}
 Since the assumptions are stable under desingularized tame coverings (see Proposition\nobreakspace \ref {enoughtamecoveringsstable}), it suffices to show that the cokernel of
 \[
  H^0(B,R^2_D \pi_* \Lambda) \to H^0(B,R^2 \pi_* \Lambda)
 \]
 is killed by a desingularized $\mc$-covering.
 We have a direct sum decomposition indexed by the irreducible components~$D_i$ of~$D$:
 \begin{equation*}
  R^2_D \pi_* \Lambda = \bigoplus_i R^2_{D_i} \pi_* \Lambda.
 \end{equation*}
 It is sufficient to prove that the cokernel of the vertical part vanishes after a desingularized~$\mc$-covering.
 Both~$R^2_{D_v} \pi_* \Lambda$ and~$R^2 \pi_* \Lambda$ are skyscraper sheaves with support in the singular locus of~$X \to B$.
 We can treat each singular fiber separately and thus assume that~$B$ is a henselian discrete valuation ring.
 We only have to make sure that the constructed desingularized~$\mc$-covering extends to a desingularized~$\mc$-covering above the initial base scheme.
 We have the following diagram of exact sequences
 \begin{equation} \label{intersectiondiagram}
  \begin{tikzcd}
   H^0(B,R^2_{D_v} \pi_* \mZ(\mc))	\ar[r,hookrightarrow,"\cdot m"]	\ar[d,"\rho"]	& H^0(B,R^2_{D_v} \pi_* \mZ(\mc))	\ar[r,twoheadrightarrow]	\ar[d,"\rho"]	& H^0(B,R^2_{D_v} \pi_* \Lambda)	\ar[d,"\rho"]	\\
   H^0(B,R^2 \pi_* \mZ(\mc))		\ar[r,hookrightarrow,"\cdot m"]			& H^0(B,R^2 \pi_* \mZ(\mc))		\ar[r,twoheadrightarrow]			& H^0(B,R^2 \pi_* \Lambda).
  \end{tikzcd}
 \end{equation}
 The exactness of the above sequences can be checked by using the explicit description of the cohomology groups involved.

 In order to show that the cokernel of the right hand side vertical map in diagram~(\ref{intersectiondiagram}) vanishes after a desingularized~$\mc$-covering
  it suffices to show that the cokernel of the middle vertical map does so.
 The stalk of the morphism~$R^2_{D_v} \pi_* \mZ(\mc) \to R^2 \pi_* \mZ(\mc)$ at~$\bar{b}$ is
 \begin{equation*}
  H^2_{D_{\bar{b}}}(X^{\mathit{sh}}, \mZ(\mc)) \to H^2(X^{\mathit{sh}}, \mZ(\mc)).
 \end{equation*}
 By Lemma\nobreakspace \ref {intersectionmatrix} it is given by the intersection matrix~$\rho$ of the components of~$D_{\bar{b}}$.
 Since~$D_{\bar{b}}$ does not contain all components of the geometric special fiber,~$\rho$ is injective.
 Denote by~$\mathcal{F}$ the cokernel.
 By Lemma\nobreakspace \ref {cokernelvanishes} there is a desingularized~$\mc$-covering~$(\bar{X}',\bar{D}') \to (\bar{X},\bar{D})$ such that~$\mathcal{F} \to \mathcal{F}'$ is the zero map (where~$\mathcal{F}'$ is the respective cokernel defined on~$X'$).
 We have an exact sequence
 \begin{equation*}
  0 \to H^0(B,R^2_{D_v} \pi_* \mZ(\mc)) \to H^0(B,R^2 \pi_* \mZ(\mc)) \to H^0(B,\mathcal{F}).
 \end{equation*}
 So the cokernel of~$H^0(B,R^2_{D_v} \pi_* \mZ(\mc)) \to H^0(B,R^2 \pi_* \mZ(\mc))$ is a subgroup of~$\mathcal{F}$.
 This shows the result.
\end{proof}

\section{Exceptional fibers} \label{exceptionalfibres}

Remember that we postponed the discussion of three issues:
Firstly, we have to show that the category $\mathfrak{I}_{(\bar{X},\bar{D})}$ is cofiltered.
Secondly, we have yet to see that the dual graph of the exceptional divisor of a desingularized tame covering is simply connected.
Finally, we need that the property of having enough tame coverings is stable under desingularized $\mc$-coverings.
All three assertions rely upon an examination of the exceptional fibers of a desingularized tame covering.
In this section we describe the structure of these exceptional fibers and answer the first two questions.
The treatment of the third question is completed in Section\nobreakspace \ref {arithmeticsurfaceswithenoughtamecoverings}.

Let us call curve a noetherian scheme whose irreducible components are one-dimensional.
We say that a curve~$C$ is a \emph{chain of~$\mP^1$'s}
 if it is a scheme of finite type over a field~$k$ whose irreducible components $C_1,\ldots,C_n$ are isomorphic to~$\mP^1_k$,
 for~$i=1,\ldots,n-1$ the curve~$C_i$ intersects~$C_{i+1}$ in exactly one point, which is moreover~$k$-rational, and~$C_i \cap C_j$ is empty for~$|i-j| \geq 2$.
If~$C$ is a closed subscheme of another curve~$C_0$, we say that~$C$ is a \emph{bridge of~$\mP^1$'s} in~$C_0$ if~$C$ is a chain of~$\mP^1$'s
 and~$C$ intersects exactly two of the remaining irreducible components of~$C_0$
 and this intersection takes place in two~$k$-rational points $c_1 \in C_1$ and $c_n \in C_n$ in the smooth locus of~$C_{\red}$.

The singularities of complex surfaces arising by taking a covering whose branch locus is an snc divisor are known to be Hirzebruch-Jung singularities, \ie of type~$A_{n,q}$
 (see \cite{BHPV}, II, Thm.~5.2).
This type of singularity was studied by Hirzebruch in \cite{Hi53}.
When working with arbitrary noetherian, normal surfaces and only allowing \emph{tame} coverings, the situation is basically the same.
For arithmetic surfaces over a discrete valuation ring with algebraically closed residue field this has been worked out by Viehweg (see \cite{Vie77}).
In general, it boils down to the fact that logarithmic singularities on a surface are of type $A_{n,q}$.
This should be well known.
However, the author was not able to find a good reference.
Therefore, we include a proof.

\begin{proposition} \label{simplyconnected}
 Let~$X/B$ be an arithmetic surface and~$D \subseteq X$ a tidy divisor.
 Let~$(X_1,D_1) \to (X,D)$ be a tame covering of~$(X,D)$ and~$(X'_{\mathit{min}},D'_{\mathit{min}}) \to (X_1,D_1)$ the minimal desingularization of~$(X_1,D_1)$.
 Then~$D'_{\mathit{min}}$ is a tidy divisor and the exceptional fibers of~$X'_{\mathit{min}} \to X_1$ are bridges of~$\mP^1$'s in~$D'_{\mathit{min}}$
  (\ie the singularities of~$X_1$ at points in~$D_1$ are of type $A_n$, or Hirzebruch-Jung singularities).
 In particular,~$(X'_{\mathit{min}},D'_{\mathit{min}}) \to (X,D)$ is a tidy desingularization of~$(X,D)$.
 Moreover, for any other desingularized tame covering~$(X',D') \to (X,D)$ the exceptional fibers are bridges of~$\mP^1$'s in~$D'$, as well.
\end{proposition}

\begin{proof}
 We view~$(X,D)$ as a log scheme with log structure given by the divisor~$D$.
 Since~$D$ has normal crossings,~$(X,D)$ is log-regular and the corresponding log structure is toric.
 The tame covering $(X_1,D_1) \to (X,D)$ is Kummer \'etale by the logarithmic version of Abhyankar's lemma (see \cite{GR11}, Thm.~7.3.44).
 In particular, it is log-smooth and thus $(X_1,D_1)$ is log-regular and the corresponding log structure~$\mathscr{M}_{D_1} \to \mO_{X_1}$ toric, as well.
 In section~10 of~\cite{Kato94} Kato associates a fan~$F_{D_1}$ to the log scheme~$(X_1,D_1)$.
 In this context a fan is a monoidal space which is locally isomorphic to $\Spec P$ for a sharp monoid~$P$.
 As a topological space the fan~$F_{D_1}$ is the subspace $\{x \in X_1 | I(x,\mathscr{M}_{D_1}) = \mathfrak{m}_x\}$ of~$X_1$, where~$I(x,\mathscr{M}_{D_1})$ is the ideal generated by~$\mathscr{M}_{D_1,x}\backslash \mathscr{M}_{D_1,x}^{\times}$.
 The structure sheaf is given by the inverse image of~$\mathscr{M}_{D_1} \backslash \mO_{X_1}^{\times}$.
 Since the log structure of~$(X_1,D_1)$ is toric, the fan~$F_{D_1}$ corresponds to a classical fan~$\Delta$, i.e. a fan of convex polyhedral cones in a two-dimensional lattice~$L$ as described in~\cite{Ful93}.
 We may work locally and thus assume that~$\Delta$ consists of one two-dimensional cone~$\sigma$ together with its two one-dimensional faces~$\tau$ and~$\tau'$ and~$\{0\}$.
 The faces~$\tau$ and~$\tau'$ correspond to prime divisors~$P$ and~$P'$ on~$X_1$ constituting the irreducible components of~$D_1$ (see~\cite{Kato94}, Corollary~11.8).
 They intersect in one point~$x_1 \in X_1$, which is the only possibly singular point of~$X_1$.
 By~\cite{Ful93},~section~2.6 we can find a subdivision~$\Delta'$ of~$\Delta$ in cones which are isomorphic to~$\mN^2$.
 In dimension~$2$ a subdivision of~$\sigma$ is given by inserting additional rays~$\tau_1,\ldots,\tau_{n-1}$ forming the faces of cones~$\sigma_1,\ldots,\sigma_n$.
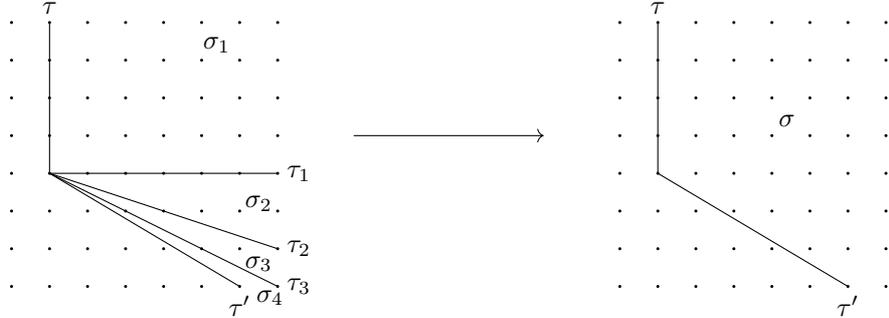
 \begin{center}
 \begin{tikzpicture}
  \foreach \x in {-0.5,0,0.5,1,1.5,2,2.5,3}
   \foreach \y in {-1.5,-1,-0.5,0,0.5,1,1.5,2}
    {
     \fill (\x,\y) circle (0.02cm);
     \fill (\x+8,\y) circle (0.02cm);
    }
    \draw (0,0)--(0,2.0)      node[above] {$\tau$};
    \draw (0,0)--(3,0)        node[right] {$\tau_1$};
    \draw (0,0)--(3,-1)       node[right] {$\tau_2$};
    \draw (0,0)--(3,-1.5)     node[right] {$\tau_3$};
    \draw (0,0)--(2.5,-1.5)   node[below] {$\tau'$};
    \draw (8,0)--(8,2.0)      node[above] {$\tau$};
    \draw (8,0)--(10.5,-1.5)  node[below] {$\tau'$};
    \draw (2.2,1.7) node {$\sigma_1$};
    \draw (2.75,-0.4) node {$\sigma_2$};
    \draw (2.75,-1.2) node {$\sigma_3$};
    \draw (2.9,-1.65) node {$\sigma_4$};
    \draw (9.7,0.7) node {$\sigma$};
    \draw[->] (4,0.5)--(6.5,0.5);
 \end{tikzpicture}
 \end{center}
 By~\cite{Kato94},~10.4. this provides us with a resolution of singularities~$(X',D') \to (X_1,D_1)$ such that~$D'$ has strictly normal crossings.
 The exceptional fiber consists of prime divisors~$E_1,\ldots,E_{n-1}$ corresponding to the rays~$\tau_1,\ldots,\tau_{n-1}$ and~$E_i$ intersects~$E_{i+1}$ in one point corresponding to the cone~$\sigma_i$.
 The strict transforms of~$P$ and~$P'$ correspond to the rays~$\tau$ and~$\tau'$.
 Hence, $P$ intersects~$\tau_1$ in one point and~$P'$ intersects~$\tau_{n-1}$ in one point.
 It remains to see that the~$E_i$ are rational.
 By the proof of~\cite{Kato94}, Proposition~9.9 we have
 $$
 X' = X_1 \times_{\mZ[\Delta]} \Spec \mZ[\Delta'].
 $$
 The exceptional fiber is thus given by
 $$
 \Spec k(x_1) \times_{\mZ[\Delta]} \Spec \mZ[\Delta'].
 $$
 Locally this is the spectrum of~$k(x_1)[\sigma_i^{\vee}/\sigma^{\vee}]$, which is readily checked to be rational.
\end{proof}

\begin{corollary} \label{coefficients}
 In the situation of Proposition\nobreakspace \ref {simplyconnected} let~$x_1$ be a special point of~$D_1$.
 Let~$Z_1$ be an irreducible component of~$D_1$ containing~$x_1$.
 Denote by~$Z'$ its strict transfer in~$X'_{\mathit{min}}$ and by~$Z$ its image in~$X$.
 Let~$E_1,\ldots,E_n$ be the irreducible components of the exceptional fiber of $X'_{\mathit{min}} \to X_1$ above~$x_1$
  such that~$E_i$ intersects~$E_{i+1}$ and~$Z'$ intersects~$E_1$.
 Then above an open neighborhood of~$x_1$ the pullback of~$Z$ to~$X'_{\mathit{min}}$ is given by
 $$
 a_0 Z' + a_1 E_1 + \ldots a_n E_n
 $$
 with~$a_0 > a_1 > \ldots > a_n > 0$.
\end{corollary}

\begin{proof}
 Denote by~$b$ the image of~$x$ in~$B$ and by~$\varphi$ the morphism $X'_{\mathit{min}} \to X$.
 In order to simplify notation, we set~$E_0 := Z'$.
 By the projection formula and Proposition\nobreakspace \ref {simplyconnected} we have
 $$
 0 = \varphi^*Z \cdot E_n = (a_0 E_0 + a_1 E_1 + \ldots a_n E_n) \cdot E_n = [k(x_1):k(b)]a_{n-1} + a_n E_n^2.
 $$
 Since the desingularization~$X'_{\mathit{min}} \to X_1$ is minimal,~$E_n$ cannot be a~$(-1)$-curve and thus~$E_n^2 < -[k(x_1):k(b)]$.
 (The self-intersection of~$E_n$ has to be negative by~\cite{Liu}, chapter~9, Theorem~1.27.)
 Hence,
 $$
 a_{n-1} = - a_n E_n^2 > a_n.
 $$
 By induction we may assume that~$a_{i+1} < a_i$ for~$0 < k \leq i < n$.
 Again by the projection formula we obtain
 $$
 0 = \varphi^*Z \cdot E_k = [k(x_1):k(b)](a_{k-1} + a_{k+1}) + a_k E_k^2.
 $$
 By induction and using~$E_k^2 \le -2[k(x_1):k(b)]$ we conclude that
 \begin{equation*}
 a_{k-1} = -a_{k+1} - \frac{a_k}{[k(x_1):k(b)]} E_k^2 \ge -a_{k+1} + 2a_k > a_k.
 \end{equation*}
\end{proof}

\begin{corollary} \label{twointersectionpoints}
 Let~$X/B$ be an arithmetic surface and~$D \subseteq X$ a tidy divisor.
 Let~$(X_1,D_1) \to (X,D)$ be a tame covering of~$(X,D)$ and~$(X',D') \to (X_1,D_1)$ a desingularization of~$(X_1,D_1)$.
 Assume that every irreducible component of an exceptional fiber of~$(X',D') \to (X_1,D_1)$ intersects the other irreducible components of~$D'$ in at least two points.
 Then~$(X',D') \to (X_1,D_1)$ is a tidy desingularization.
\end{corollary}

\begin{proof}
 We can factor~$(X',D') \to (X_1,D_1)$ as
 $$
 (X',D') := (X'_n,D'_n) \to \ldots \to (X'_1,D'_1) \to (X'_0,D'_0) \to (X_1,D_1),
 $$
 where~$(X'_0,D'_0) \to (X_1,D_1)$ is the minimal desingularization of~$(X_1,D_1)$ and for~$i=1,\ldots,n$ the morphism~$(X'_i,D'_i) \to (X'_{i-1},D'_{i-1})$ is the blowup of~$X'_{i-1}$ in a closed point~$d'_{i-1}$ of~$D'_{i-1}$.
 By Proposition\nobreakspace \ref {simplyconnected} the minimal desingularization~$(X'_0,D'_0) \to (X_1,D_1)$ is a tidy desingularization.
 Moreover, blowing up in closed points does not destroy the tidiness of a divisor.
 Hence,~$D'_i$ is a tidy divisor of~$X'_i$ for all~$i=0,\ldots n$.
 Suppose that~$(X',D') \to (X_1,D_1)$ is not a tidy desingularization.
 Then there is an index~$i$ such that~$d'_{i-1}$ is not a special point of~$D'_{i-1}$, \ie,~$d'_{i-1}$ is a regular point of~$D'_{i-1}$.
 Let~$i_0$ be the biggest such index.
 Then the exceptional fiber of~$(X'_{i_0},D'_{i_0}) \to (X'_{i_0-1},D'_{i_0-1})$ has only one intersection point with the other irreducible components of~$D'_{i_0}$.
 This does not change by blowing up~$D'_{i_0}$ in special points.
 We thus obtain a contradiction.
\end{proof}

Let~$D$ be a tidy divisor on an arithmetic surface~$X$ and~$\bar{x} \to U = X-D$ a geometric point.
The explicit description of the exceptional fibers in Proposition\nobreakspace \ref {simplyconnected} enables us to prove that the category~$\mathfrak{I}_{(X,D)}$ is cofiltered:

\begin{proposition} \label{cofiltered}
 The following assertions hold:
 \begin{enumerate}[label=(\roman*)]
  \item If~$(X',D') \to (X,D)$ and~$(X'',D'') \to (X',D')$ are both desingularized~$\mc$-coverings, the composite~$(X'',D'') \to (X,D)$ is again a desingularized~$\mc$-covering.
  \item If~$(X',D') \to (X,D)$ and~$(X'',D'') \to (X,D)$ are desingularized $\mc$-coverings, there is a commutative diagram of desingularized $\mc$-coverings
	\begin{center}
	 \begin{tikzcd}
					& (X',D')	\ar[rd]	&		\\
	  (X''',D''')	\ar[ru]	\ar[rd]	&			& (X,D).	\\
					& (X'',D'')	\ar[ru]	&
	 \end{tikzcd}
	\end{center}
 \end{enumerate}
\end{proposition}

\begin{proof}
 (i). Let~$X_1$ be the normalization of~$X$ in~$K(X')$ and~$X_2$ its normalization in~$K(X'')$.
 Furthermore, denote by~$X'_1$ the normalization of~$X'$ in~$K(X'')$.
 We obtain a Cartesian diagram
  \begin{center}
   \begin{tikzcd}
    D''	\ar[r]	\ar[d,hook]	& D_2	\ar[r]	\ar[d,hook]	& D	\ar[d,hook]	\\
    X''	\ar[r]			& X_2	\ar[r]			& X.
   \end{tikzcd}
  \end{center}
 Since~$U'=X'-D'$ is the normalization of~$U=X-D$ in~$K(X')$ and $U''=X''-D''$ is the normalization of~$U'$ in~$K(X'')$,
  we conclude that~$U''$ is also the normalization of~$U$ in~$K(X'')$.
 It is thus an open subscheme of~$X_2$ and~$U'' \to U$ is a finite \'etale~$\mc$-covering as finite \'etale~$\mc$-coverings are stable under composition.
 Hence,~$X''\to X_2$ is birational and an isomorphism on~$U''$.
 Moreover,~$D'' \subseteq X''$ is a tidy divisor.
 The only remaining question is whether~$X'' \to X_2$ is obtained from the minimal desingularization of~$(X_2,D_2)$ by successively blowing up in special points.
 By Corollary\nobreakspace \ref {twointersectionpoints} it suffices to show that every irreducible component of an exceptional fiber of~$X'' \to X_2$ meets the other irreducible components of~$D''$ in at least two points.
 The morphisms~$X'' \to X'$ and~$X' \to X$ factor as
 \begin{IEEEeqnarray*}{lllll}
  (X',D')   &= (Y_m,Z_m)		&\to \ldots \to (Y_0,Z_0) &= (X_1,D_1)   &\to (X,D),	\\
  (X'',D'') &= (Y_n,Z_n)	&\to \ldots \to (Y_{m+1},Z_{m+1}) &= (X'_1,D'_1) &\to (X',D'),
 \end{IEEEeqnarray*}
 where $(Y_0,Z_0) \to (X_1,D_1)$ and $(Y_{m+1},Z_{m+1}) \to (X'_1,D'_1)$ are minimal desingularizations of~$(X'_1,D'_1)$ and~$(X_1,D_1)$, respectively,
  and for $i=1,\ldots,m$ and~$i=m+2,\ldots,n$ the morphism~$(Y_i,Z_i) \to (Y_{i-1},Z_{i-1})$ is the blowup of~$Y_{i-1}$ in a special point~$z_{i-1}$ of~$Z_{i-1}$.
 Let~$E$ be an irreducible component of an exceptional fiber of~$X'' \to X_2$.
 There is~$i \in \{1,\ldots,m\} \cup \{m+2,\ldots n\}$ such that the image of~$E$ in~$Y_i$ is one-dimensional and its image in~$Y_{i-1}$ is a closed point.
 This closed point is precisely the point~$z_{i-1}$ and we obtain a finite morphism from~$E$ to the exceptional fiber of~$Y_i \to Y_{i-1}$.
 Since~$X_i \to X_{i-1}$ is the blowup of~$X_{i-1}$ in~$z_{i-1}$ and~$z_{i-1}$ is a special point, its exceptional fiber intersects the other irreducible components of~$Z_i$ in two points.
 The intersection points of~$E$ contain the preimages of these two points and thus there are at least two intersection points.

 (ii). Let~$K'''$ be the compositum of~$K(X')$ and~$K(X'')$ and~$X_3$ the normalization of~$X$ in~$K'''$.
 This defines a~$\mc$-covering~$(X_3,D_3) \to (X,D)$.
 We obtain rational maps~$X_3 \dashrightarrow X'$ and~$X_3 \dashrightarrow X''$, which, restricted to~$U_3=X_3-D_3$, are finite \'etale~$\mc$-coverings of~$U'=X'-D'$ and~$U''=X''-D''$, respectively.
 Using elimination of indeterminacies and the existence of tidy desingularizations we find a desingularization~$(X''',D''') \to (X_3,D_3)$ dominating~$(X',D')$ and~$(X'',D'')$ such that~$D'''$ is tidy.
 Suppose there is an irreducible component~$E$ of an exceptional fiber of~$X'''$ with only one intersection point with the other irreducible components of~$D'''$.
 Let us write
 $$
 (X''',D''') = (X'''_n,D'''_n) \to \ldots \to (X'''_0,D'''_0) \to (X_3,D_3),
 $$
 where~$(X'''_0,D'''_0) \to (X_3,D_3)$ is the minimal desingularization of~$(X_3,D_3)$ and for~$i=1,\ldots,n$ the morphism~$X'''_i \to X'''_{i-1}$ is the blowup of~$X'''_{i-1}$ in a closed point~$d_{i-1} \in D'''_{i-1}$.
 There is~$i \in \{1,\ldots,n\}$ such that the image of~$E$ is the point~$d_{i-1}$ and the image of~$E$ in~$X'''_i$ is the exceptional fiber~$E_i$ of~$X'''_i \to X'''_{i-1}$.
 Since~$E$ has only one intersection point, the same holds for~$E_i$.
 Furthermore, the blowup points~$d_{k-1}$ for~$k=i+1, \ldots,n$ must not lie above~$E_i$ except possibly above the intersection point~$z_{i-1}$ of~$E_{i-1}$ with the other irreducible components.
 One checks that after blowing up in~$z_{i-1}$ the strict transform of~$E_i$ is still a~$(-1)$-curve.
 Therefore, we can contract~$E$.
 Moreover, by similar arguments as in the proof of part~(i) the image of~$E$ in~$X'$ as well as in~$X''$ is a point.
 Hence, the contraction still factors through~$X' \to X$ and~$X'' \to X$.
 After finitely many contractions we may assume that all irreducible components of exceptional fibers of~$X''' \to X_3$ have at least two intersection points.
 Then the same holds for the exceptional fibers of~$X''' \to X'$ and of~$X''' \to X''$ as these are contained in the exceptional fibers of~$X''' \to X_3$.
 The assertion now follows from Corollary\nobreakspace \ref {twointersectionpoints}.
\end{proof}

\section{Stability of enough tame coverings} \label{arithmeticsurfaceswithenoughtamecoverings}

Let us fix an arithmetic surface $X/B$ and a tidy divisor $D \subseteq X$.
The aim of this section is to show:

\begin{proposition} \label{enoughtamecoveringsstable}
 Let~$\pi:(X',D') \to (X_1,D_1) \to (X,D)$ be a desingularized $\mc$-covering.
 If $(X,D)$ has enough tame coverings, the same holds for~$(X',D')$.
\end{proposition}

For the proof of Proposition\nobreakspace \ref {enoughtamecoveringsstable} we need to investigate the multiplicities of the irreducible components of the pullback to~$X'$ of a prime divisor on~$X$.

\begin{definition} \label{multiplicitymatrix}
 Let~$f:(X',D') \to (X,D)$ be a desingularized $\mc$-covering.
 Let~$x' \in D'$ be a closed point and denote by~$x \in D$ the image of~$x'$ in~$X$.
 Let us call~$D_1,\ldots,D_n$ (necessarily $n=1$ or~$n=2$) the irreducible components of~$D$ passing through~$x$
  and~$D'_1,\ldots,D'_m$ ($m \leq n$) the irreducible components of~$D'$ passing through~$x'$.
 Restricting~$f$ to a suitable neighborhood of~$x'$, the pullback of Cartier divisors via~$f$ induces a homomorphism
 $$
 \mQ \cdot D_1 \oplus \ldots \oplus \mQ \cdot D_n \to \mQ \cdot D'_1 \oplus \ldots \oplus \mQ \cdot D'_m.
 $$
 We call this morphism \emph{multiplicity homomorphism} at~$x'$ and its transformation matrix with respect to the above bases \emph{multiplicity matrix} at~$x'$.
\end{definition}

Multiplicity homomorphisms are compatible with composition.
If~$(X'',D'') \to (X',D')$ is another morphism as above and~$x''$ a closed point of~$D''$ mapping to~$x' \in D'$, the multiplicity homomorphism of~$(X'',D'') \to (X',D')$ at~$x''$
 is the composition of the multiplicity homomorphism of~$(X'',D'') \to (X',D')$ at~$x''$ and the multiplicity homomorphism of~$(X',D') \to (X,D)$ at~$x'$.

\begin{lemma} \label{regularblowup}
 Let~$(X',D') \to (X,D)$ be the blowup of~$X$ in a special point~$x$ of~$D$.
 Then all multiplicity homomorphisms are surjective.
\end{lemma}

\begin{proof}
 Denote by~$D_1$ and~$D_2$ the irreducible components of~$D$ passing through~$x$ and by~$D'_1$ and~$D'_2$ their strict transforms in~$X'$.
 Furthermore, let~$E$ denote the singular fiber of~$X' \to X$.
 On~$E \subseteq D'$ there are two points~$x'_1$ and~$x'_2$ where~$D'$ is singular, namely the respective intersection points with~$D'_1$ and~$D'_2$.
 The pullback of~$D_i$ is given by~$D'_i + E$.
 Hence, the intersection matrix at~$x'_1$ as well as at~$x'_2$ (with respect to the bases~$\{(D_1,D_2),(D'_1,E)\}$ and~$\{(D_1,D_2),(E,D'_2)\}$, respectively) is
 $$
 \begin{pmatrix}
  1 & 0 \\
  1 & 1
 \end{pmatrix},
 $$
 which is invertible.
 If~$x' \in E$ is a nonsingular point of~$D'$, its multiplicity matrix is
 $$
 \begin{pmatrix}
  1 & 1
 \end{pmatrix},
 $$
 which is nonzero and thus its multiplicity homomorphism is surjective.
 The multiplicity homomorphism at any other closed point of~$D'$ is the identity.
\end{proof}

\begin{lemma} \label{multiplicitysurjective}
 Let~$\varphi:(X',D') \to (X_1,D_1) \to (X,D)$ be a desingularized $\mc$-covering.
 Then all multiplicity homomorphisms are surjective.
\end{lemma}

\begin{proof}
 By Lemma\nobreakspace \ref {regularblowup} we may assume that~$X' \to X_1$ is the minimal desingularization of~$X_1$.
 Let~$x' \in D'$ be a closed point and denote by~$x_1$ and~$x$ the image of~$x'$ in~$X_1$ and~$X$, respectively.
 If~$x'$ is a regular point of~$D'$, there is only one irreducible component of~$D'$ passing through~$x'$.
 Hence, the multiplicity homomorphism at~$x'$ is surjective if and only if it is nonzero, which is clear by taking the pullback of any irreducible component of~$D$ passing through~$x$.

 Suppose that~$x'$ is a singular point of~$D'$.
 Then also~$x_1$ and~$x$ are singular points of~$D_1$ and~$D$, respectively.
 There are two irreducible components~$Z_1$ and~$W_1$ of~$D_1$ passing through~$x_1$ mapping to the irreducible components~$W$ and~$Z$ of~$D$ passing through~$x$.
 According to Corollary\nobreakspace \ref {coefficients} we have in a neighborhood of~$x'$
 $$
 \varphi^*Z = a_0 Z' + a_1 E_1 + \ldots a_n E_n
 $$
 with~$a_0 > a_1 > \ldots a_n > 0$ and
 $$
 \varphi^*W = b_1 E_1 + \ldots b_n E_n + b_{n+1} W'
 $$
 with~$b_1 < \ldots < b_n < b_{n+1}$, and where~$Z'$ and~$W'$ denote the strict transforms of~$Z_1$ and~$W_1$, respectively, in~$X'$.
 Setting~$E_0 := Z'$ and~$E_{n+1} := W'$ we know that there is an integer~$i$ with~$0 \leq i \leq n$ such that~$x'$ is the intersection point of~$E_i$ with~$E_{i+1}$.
 The multiplicity matrix at~$x'$ is
 $$
 \begin{pmatrix}
  a_i     & b_i      \\
  a_{i+1} & b_{i+1}
 \end{pmatrix}
 $$
 and
 $$
 \mathit{det}  \begin{pmatrix}
		a_i     & b_i      \\
		a_{i+1} & b_{i+1}
	       \end{pmatrix}		= a_i b_{i+1} - a_{i+1} b_i > a_i b_i - a_i b_i = 0
 $$
 as~$a_{i+1} < a_i$ and~$b_{i+1} > b_i$.
 Therefore, also in this case the multiplicity homomorphism is surjective.
\end{proof}

\begin{proof}[Proof of Proposition\nobreakspace \ref {enoughtamecoveringsstable}]
 Assume that~$(X,D)$ has enough tame coverings.
 Let $x' \in D'$ be a closed point and~$Z'$ an irreducible component of~$D'$ passing through~$x'$.
 We have to find $f' \in K(X')^{\times}$ with support in~$D'$ such that $\deg_{Z'}(f') > 0$
  and~$\deg_{C'}(f') = 0$ for all other irreducible components~$C'$ of~$D'$ passing through~$x'$.
 Let~$Z_1,\ldots,Z_r$ (for~$r=1$ or~$r=2$) denote the irreducible components of~$D$ passing through the image point~$x \in D$ of~$x'$.
 Since~$(X,D)$ has enough tame coverings, for~$i=1,\ldots,r$ there is~$f_i \in K(X)^{\times}$ with support in~$D$
  such that~$\deg_{Z_i}(f_i) > 0$ and~$\deg_{Z_j}(f_i) = 0$ for~$i \neq j$.
 The projections of~$div~f_i$ to
 $$
 \mQ \cdot Z_1 \oplus \ldots \oplus \mQ \cdot Z_r
 $$
 constitute a basis of this vector space.
 Let~$Z' = Z'_1,\ldots,Z'_s$ denote the irreducible components of~$D'$ passing through~$x'$.
 Lemma\nobreakspace \ref {multiplicitysurjective} provides the surjectivity of the multiplicity homomorphism
 $$
 \phi_{x'}:\mQ \cdot Z_1 \oplus \ldots \oplus \mQ \cdot Z_r \to  \mQ \cdot Z'_1 \oplus \ldots \oplus \mQ \cdot Z'_s
 $$
 at~$x'$ induced by pullback.
 We obtain integers~$d,k_1,\ldots,k_r$ with~$d > 0$ such that
 $$
 d \cdot Z'_1 = \phi_{x'}(k_1 div~f_1 + \ldots k_r div~f_r).
 $$
 In other words, setting~$f = f_1^{k_1} \cdot \ldots f_r^{k_r}$ we have in a neighborhood of~$x'$
 $$
 div~f = d \cdot Z',
 $$
 what we wanted to prove.
\end{proof}

\section{Neighborhoods with enough tame coverings} \label{neighborhoodswithenoughtamecoverings}

At this point we have completed the discussion of conditions~(1) and~(2) in Proposition\nobreakspace \ref {firstreductions}.
As a result we know that under the assumptions listed in Proposition\nobreakspace \ref {etalecovering} and Proposition\nobreakspace \ref {lerayfiltration} the arithmetic surface~$U$ is $K(\pi,1)$ with respect to~$\mc$.
The remaining task is to construct neighborhoods on a given arithmetic surface satisfying these assumptions.
The property of having enough tame coverings is the most difficult to realize.
It is the aim of this section to explain how to construct neighborhoods with enough tame coverings.

We use the following notation:
For an integral closed subscheme~$Z$ of an affine scheme~$Spec~A$ we denote by~$\mathfrak{p}_Z$ the prime ideal of~$A$ corresponding to the generic point of~$Z$.
Moreover,~we write~$m_x(Z)$ for the multiplicity of a closed subscheme~$Z$ in a point~$x$.

\begin{lemma} \label{picardgenerators}
 Let~$X/B$ be a quasi-projective arithmetic surface such that~$B$ is a discrete valuation ring with finitely generated quotient field.
 Let~$\mathscr{M}$ be a finite subset of~$X$ containing all singular points.
 Then there are horizontal prime Cartier divisors~$G_1,\ldots,G_s$, $G_{s+1},\ldots,G_r$
  such that~$G_1,\ldots,G_s$ and~$G_{s+1},\ldots,G_r$ each generate $\text{Pic}(X)$.
 Furthermore, the supports of~$G_i$ for~$i=1,\ldots,r$ do not contain any $x \in \mathscr{M}$ and the supports of~$G_i$ and~$G_j$ for~$i \leq s$ and~$j > s$ are disjoint.
\end{lemma}

\begin{proof}
 The generic fiber~$X_{\eta}$ of~$X \to B$ is a smooth curve over a finitely generated field.
 By a generalization of the Mordell-Weil theorem due to N\'eron (see~\cite{Ne52}) its Weil divisor class group is finitely generated.
 Denote by~$C_1,\ldots,C_l$ the irreducible components of the special fiber.
 The Weil divisor class group $CH^1(X)$ of~$X$ is generated by the Weil divisor class group of~$X_{\eta}$ and by~$C_1,\ldots,C_l$.
 It is therefore also finitely generated.
 As~$X$ is normal, the Picard group of~$X$ injects into $CH^1(X)$.
 It is thus generated by finitely many Cartier divisors~$D_1,\ldots,D_m$.
 
 Since~$X$ is quasi-projective over an affine scheme, there is an affine open subscheme of~$X$ containing $\mathscr{M}$
  and all generic points of the supports of $D_1,\ldots,D_m$ and $C_1,\ldots,C_l$ (see~\cite{Liu}, Proposition~3.3.36).
 Taking the limit over all of these affine open subschemes, we obtain the spectrum of a semi-local ring~$A$.
 As the Picard group of semi-local schemes is trivial, we may find $f_j \in A$ for $j=1,\ldots,m$ such that
 \[
  \Div f_j = D_j
 \]
 on $\Spec A$.
 Viewing~$f_j$ as elements of~$K(X)^{\times}$ we obtain horizontal Cartier divisors $D_j'=  D_j-div~f_j$ for $j=1,\ldots m$ generating $\Pic(X)$
  whose supports are disjoint from~$\mathscr{M}$.
 Denote by~$G_1,\ldots,G_s$ the prime divisors in the union of the supports of~$D_1',\ldots,D_m'$.
 Since all singular points of~$X$ are contained in~$\mathcal{M}$, these are in fact Cartier divisors.

 Denote by~$\mathscr{M}'$ the union of~$\mathscr{M}$ with the supports of~$G_1,\ldots,G_s$ (which is a finite set).
 By the same argument as above we find horizontal prime Cartier divisors $G_{s+1},\ldots,G_r$ generating $\Pic(X)$ whose supports are disjoint from~$\mathscr{M}'$.
 Hence, the support of~$G_i$ for~$i \leq s$ is disjoint from the support of~$G_j$ for~$j > s$.
\end{proof}

\begin{lemma} \label{blowupgenerators}
 Let~$X/B$ be an arithmetic surface and let~$G_1,\ldots,G_s$ be horizontal prime divisors generating the Picard group $\Pic(X)$.
 Let~$x$ be a closed point of~$X$ of codimension~$2$ such that~$X$ is regular at~$x$ and~$x$ is not contained in any~$G_j$ for~$j=1,\ldots,s$.
 Denote by~$X' \to X$ the blowup of~$X$ in~$x$.
 Let~$G$ be a horizontal prime Cartier divisor on~$X'$ disjoint from~$G_1,\ldots,G_s$ with nontrivial intersection with the exceptional locus~$E$.
 Then~$G_1,\ldots,G_s,G$ generate~$\Pic(X') \otimes \mQ$.
\end{lemma}

\begin{proof}
 The Picard group of~$X'$ is generated by~$G_1,\ldots,G_s$ and~$E$.
 Let~$G_0$ denote the image of~$G$ in~$X$.
 Then~$G_0$ is a Cartier divisor as~$x$ is regular.
 Since~$G_1,\ldots,G_s$ generate the Picard group of~$X$, there are~$n_j \in \mZ$ such that
 $$
 G_0 = \sum_{j=1}^s n_j G_j
 $$
 in~$\Pic(X)$.
 By~\cite{Liu}, Chapter~9, Proposition~2.23 the pullback of~$G_0$ to~$X'$ is given by
 $$
 G + m_x(G_0) \cdot E.
 $$
 Since~$x \in G_0$, the multiplicity~$m_x(G_0)$ is positive.
 In~$\Pic(X') \otimes \mQ$ we thus have
 \begin{equation*}
 E = \frac{1}{m_x(G_0)}(\sum_{j=1}^s n_j G_j - G).
 \end{equation*}
\end{proof}

\begin{lemma} \label{horizontaldivisors}
 Let~$X$ be a projective arithmetic surface over a discrete valuation ring~$B$ and $D$ a tidy divisor on~$X$.
 Let~$x,z_1,\ldots,z_k \in X$ be closed points such that~$X$ and the reduced special fiber~$X_{s,\red}$ are regular at~$x$.
 Assume moreover that~$x$ is not a special point of~$D$.
 Then there is a horizontal Cartier prime divisor~$D_x$ passing through~$x$ and disjoint from $z_1,\ldots,z_k$ such that $D_x + D$ is tidy.
\end{lemma}

\begin{proof}
 Denote by $s = \Spec k$ the special point of~$B$ and by $\eta = \Spec K$ the generic point.
 Choose an embedding $X \hookrightarrow \mP^N_B$.
 This induces embeddings $X_s \hookrightarrow \mP^N_k$ and $X_{\eta} \hookrightarrow \mP^N_K$.
 Let~$T$ be the finite subscheme of $\mP^N_k$ which is the disjoint union of all singular points of~$X$, all singular points of~$X_{s,\red}$ and all special points of~$D$ (they are all contained in the special fiber).
 In order to prove the lemma it suffices to find a hyperplane~$H$ of~$\mP_B^N$ intersecting~$X$ transversally, passing through~$x$, and disjoint from~$T$ such that $D_x := H \times_{\mP_B^N} X$ is regular and $D_x + D$ is tidy.
 By \cite{JS12}, Lemma~1.3 a hyperplane~$H$ satisfies these conditions if
 \begin{enumerate}[label=(\roman*)]
  \item $H_s$ intersects~$X_s$ transversally, passes through~$x$, and is disjoint from~$T$,
  \item $H_{\eta}$ intersects~$X_{\eta}$ transversally.
 \end{enumerate}
 
 Assume first that~$k$ is finite.
 By \cite{Poo04}, Thm.~1.2 there is a hypersurface~$H_s$ of~$\mP^N_k$ intersecting~$X_s$ transversally, passing through~$x$, and disjoint from~$T$.
 Changing the projective embedding we may assume that~$H_s$ is a hyperplane.
 If~$k$ is infinite, the existence of the hyperplane~$H_s$ follows by the classical Bertini theorem.
 
 Let~$H$ be any hyperplane of~$\mP_B^N$ with special fiber the hyperplane~$H_s$ constructed above.
 We claim that the generic fiber~$H_{\eta}$ intersects~$X_{\eta}$ transversally.
 Let $y \in X_{\eta}$ be a closed point in the intersection and choose a point~$y_s$ of the special fiber which is a specialization of~$y$.
 Then~$y_s$ is not contained in~$T$ as~$H_s$ is disjoint from~$T$.
 Hence, $H_s$ intersects~$X_s$ transversally at~$y_s$.
 Since~$y$ is a generalization of~$y_s$ this implies that~$H_{\eta}$ intersects~$X_{\eta}$ transversally at~$y$.
\end{proof}

\begin{proposition} \label{opensubscheme}
 Let~$Y/B$ be an arithmetic surface such that~$B$ is a discrete valuation ring with finitely generated quotient field and~$x \in Y$ a closed point in the special fiber.
 Then there is an open neighborhood~$V \subseteq Y$ of~$x$ and a compactification~$\bar{X}/B$ of~$V$ such that~$\bar{D}=\bar{X}-V$ is a tidy divisor and such that the following assertion holds:
 For every closed point~$y \in \bar{X}$ and every prime Cartier divisor~$Z$ of~$\bar{X}$ passing through~$y$ there is~$f \in K(\bar{X})^{\times}$ with support in~$Z \cup \bar{D}$ such that~$\deg_Z(f) > 0$ and~$\deg_W(f)=0$ for all other prime divisors~$W$ passing through~$y$.
\end{proposition}

\begin{proof}
 Take an affine open neighborhood~$V'$ of~$x$ such that the complement contains all singular points except~$x$ and all vertical prime divisors not passing through~$x$.
 Since~$V'$ is affine, we can choose a projective compactification~$\bar{V}'$ of~$V'$ over~$B$.
 Set~$D'=\bar{V}'-V'$ with the reduced scheme structure.
 By~\cite{MR0491722} we can replace~$(\bar{V},D')$ by a desingularization (in the strong sense) and thus assume that~$x$ is the only possible singular point of~$\bar{V}'$ and~$D'$ is a Cartier divisor.
 Choose prime divisors~$G_1,\ldots,G_r$ of~$\bar{V}'$ not passing through~$x$ as in Lemma\nobreakspace \ref {picardgenerators}.
 Making~$V'$ smaller we may assume that~$G_1,\ldots,G_r$ are contained in~$D'$.

 Let~$(\bar{X},D_0) \to (\bar{V}',D')$ be a tidy desingularization, which exists by Proposition\nobreakspace \ref {tidydesingularization}.
 Since~$\bar{V}'$ is regular at every point in~$D'$, the morphism~$\bar{X} \to \bar{V}'$ is a consecutive blowup in closed points of~$D'$.
 Moreover, the exceptional fiber of each blowup in a closed point~$z$ is isomorphic to~$\mP^1_{k(z)}$ (see~\cite{Liu}, Chapter~8, Theorem~1.19).
 Denote by~$E_1,\ldots,E_n$ the irreducible components of the exceptional divisor of~$\bar{X} \to \bar{V}'$.
 For each~$i=1,\ldots,n$ choose two different closed points~$y_i,z_i \in E_i$ in the regular locus of~$D_0$.
 By Lemma\nobreakspace \ref {horizontaldivisors} there is a (horizontal) Cartier prime divisor~$D_1$ intersecting~$E_1$ transversally at~$y_1$ and disjoint from $y_2,\ldots,y_n,z_1,\ldots,z_n$ such that $D_0 + D_1$ is tidy.
 By the same argument there is a prime divisor~$D_2$ intersecting~$E_2$ transversally at~$y_2$ and disjoint from $y_3,\ldots,y_n,z_1,\ldots,z_n$ such that $D_0 + D_1 + D_2$ is tidy.
 Continuing this way we obtain for $i = 1,\ldots,n$ horizontal prime divisors $D_i$ and~$K_i$ intersecting~$E_i$ transversally at~$y_i$ and~$z_i$, respectively, and such that
 \[
  \bar{D} := D_0 + D_1 + \ldots + D_n + K_1 + \ldots + K_n
 \]
 is tidy.
 We set~$V = \bar{X} - \bar{D}$.

 We claim that~$(\bar{X},\bar{D})$ has the required properties.
 Let~$y \in \bar{X}$ be a closed point and~$Z$ a Cartier prime divisor of~$\bar{X}$ passing through~$y$.
 Either $G_1,\ldots,G_s$ or $G_{s+1},\ldots,G_r$ do not pass through~$y$, say $G_1,\ldots,G_s$.
 Similarly, either $D_1,\ldots,D_n$ or $K_1,\ldots,K_n$ do not pass through~$y$, say $D_1,\ldots,D_n$.
 By Lemma\nobreakspace \ref {blowupgenerators} the prime divisors $G_1,\ldots,G_s$, $D_1,\ldots,D_n$ generate $\Pic(\bar{X}) \otimes \mQ$.
 Hence, there are integers $m$, $m_1,\ldots,m_n$, $n_1,\ldots,n_s$ with~$m > 0$ and~$f \in K(\bar{X})^{\times}$ such that
 $$
 mZ = \sum_{j=1}^n m_j D_j + \sum_{j=1}^s n_j G_j + div~f.
 $$
 The prime divisors~$D_1,\ldots,D_n$ and~$G_1,\ldots,G_s$ do not pass through~$y$.
 Therefore,~$\deg_W(f) = 0$ for all prime divisors~$W$ different from~$Z$ passing through~$y$ and~$\deg_Z(f) = m > 0$.
 Furthermore,~$D_1,\ldots,D_n,G_1,\ldots,G_s$ are contained in~$\bar{D}$ and thus~$f$ has support in~$Z \cup \bar{D}$.
\end{proof}

As a direct consequence of Proposition\nobreakspace \ref {opensubscheme} we obtain:

\begin{corollary} \label{openwithenoughtamecoverings}
 In the situation of Proposition\nobreakspace \ref {opensubscheme} let~$U \subseteq V$ be a neighborhood of~$x$ such that~$D'=\bar{X}-U$ is the support of a tidy divisor.
 Then~$(\bar{X},D')$ has enough tame coverings.
\end{corollary}

\section{The main result} \label{mainresults}

We are now in the position to construct neighborhoods on an an arithmetic surface~$Y/B$
 satisfying all assumptions made in Proposition\nobreakspace \ref {etalecovering} and Proposition\nobreakspace \ref {lerayfiltration}.
Note that the assumption on the fundamental group of~$B$ is automatic in the local case.

\begin{theorem} \label{goodopenneighborhood}
 Let~$B$ be the spectrum of a henselian discrete valuation ring~$R$ which is formally smooth over a discrete valuation ring with finitely generated quotient field.
 Let~$\mc$ be a full class of finite groups such that the residue characteristic of~$R$ is not contained in~$\mN(\mc)$.
 Assume that $\mu_{\ell} \subseteq R$ for all primes $\ell \in \mN(\mc)$ and that the absolute Galois group of the residue field of~$R$ is $\mc$-good.
 Let~$\pi : Y \to B$ be an arithmetic surface and~$x \in Y$ a point.
 Then there is an open neighborhood~$U$ of~$x$ and a compactification~$U \subseteq \bar{X}$ of~$U \to B$
  such that the complement~$\bar{D}$ of~$U$ in~$\bar{X}$ is a tidy divisor with the following properties.
 \begin{enumerate}[label=(\roman*)]
  \item The horizontal part of~$\bar{D}$ has nontrivial intersection with all vertical prime divisors on~$\bar{X}$.
  \item $(\bar{X},\bar{D})$ has enough tame coverings.
 \end{enumerate}
 As a consequence~$U$ is $K(\pi,1)$ with respect to~$\mc$.
 \end{theorem}

\begin{proof}
 Without loss of generality we may assume that~$x$ is a closed point lying over the closed point~$b$ of~$B$.
 The arithmetic surface $Y/B$ is of finite presentation.
 Hence, it is the base change to~$B$ of an arithmetic surface~$Y_0/B_0$ such that~$B_0$ is a discrete valuation ring with finitely generated quotient field
  and~$B$ is formally smooth over~$B_0$.
 Formally smooth base change does not affect the tidiness of a divisor, nor does it disturb properties~(i) and~(ii).
 Therefore, it suffices to construct~$U$ with properties (i) and (ii) for~$B$ local with finitely generated quotient field.

 Choose an open neighborhood~$V$ of~$x$ and a compactification~$\bar{X}/B$ as in Proposition\nobreakspace \ref {opensubscheme}.
 Denote by~$D'$ the complement of~$V$ (with the reduced scheme structure).
 On every irreducible component~$C$ of~$\bar{X}_b$ take a closed point~$c_C \neq x$ in the smooth locus of~$C$ and not contained in any other irreducible component of~$\bar{X}_b$.
 Using Lemma\nobreakspace \ref {horizontaldivisors} we construct a horizontal divisor~$D''$ passing through~$c_C$ for every vertical prime divisor~$C$ such that $\bar{D} := D' + D''$ is tidy.
 Then~$(\bar{X},\bar{D})$ has enough tame coverings by Corollary\nobreakspace \ref {openwithenoughtamecoverings}.
 Moreover,~$(\bar{X},\bar{D})$ has properties~(i) and~(ii).
 
 Setting $U = \bar{X}-\bar{D}$ we conclude that~$U$ is $K(\pi,1)$ with respect to~$\mc$ by combining Proposition\nobreakspace \ref {firstreductions}, Proposition\nobreakspace \ref {etalecovering}, and Proposition\nobreakspace \ref {lerayfiltration}.
 Note that the assumptions on the roots of unity and the residue field of~$R$ are part of the general setup described in Section\nobreakspace \ref {setupandnotation}.
 They are thus implicit in Proposition\nobreakspace \ref {etalecovering} and Proposition\nobreakspace \ref {lerayfiltration}.
\end{proof}

Notice that the absolute Galois group of an algebraic extension of a finite field is $\mc$-good for any class of finite groups~$\mc$.
Moreover, completion in formally smooth.
Therefore, Theorem\nobreakspace \ref {goodopenneighborhood} implies Theorem\nobreakspace \ref {theoremintro} as stated in the introduction.
The following corollaries give more explicit examples of situations where Theorem\nobreakspace \ref {goodopenneighborhood} applies.
Remember that for given primes $\ell_1,\ldots,\ell_n$,
 we denoted by $\mc(\ell_1,\ldots,\ell_n)$ the full class of finite groups whose orders are contained in the submonoid of~$\mN$ generated by $\ell_1,\ldots,\ell_n$.

\begin{corollary}
 Let~$Y$ be an arithmetic surface over the spectrum~$B$ of a discrete valuation ring and~$y \in Y$ a point.
 Let~$\mc$ be a full class of finite groups such that the residue characteristic of~$B$ is not contained in~$\mN(\mc)$.
 Then there is a basis of Zariski neighborhoods of~$y$ which are $K(\pi,1)$ with respect to~$\mc$ in the following cases:
 \begin{enumerate}[label=(\roman*)]
  \item $B$ is the spectrum of the ring of integers of a finite extension~$K$ of $\mQ_p$
        and~$\mc$ is of the form $\mc(\ell_1,\ldots,\ell_n)$ for primes $\ell_1,\ldots,\ell_n \ne p$ such that $\mu_{\ell_i} \subseteq K$.
  \item $B$ is the spectrum of the ring of integers of the completion of the maximal unramified extension of a finite extension of~$\mQ_p$.
 \end{enumerate}
\end{corollary}

\begin{corollary}
 Let~$Y$ be an arithmetic surface over the spectrum~$B$ of a discrete valuation ring and~$\bar{y} \to Y$ a geometric point.
 Let~$\mc$ be a full class of finite groups such that the residue characteristic of~$B$ is not contained in~$\mN(\mc)$.
 Then there is a basis of \'etale neighborhoods of~$\bar{y}$ which are $K(\pi,1)$ with respect to~$\mc$ in the following cases:
 \begin{enumerate}[label=(\roman*)]
  \item $B$ is the spectrum of the ring of integers of a finite extension of $\mQ_p$ and~$\mc$ is of the form $\mc(\ell_1,\ldots,\ell_n)$ for primes $\ell_1,\ldots,\ell_n \ne p$.
  \item $B$ is the spectrum of the ring of integers of the completion of the maximal unramified extension of a finite extension of~$\mQ_p$.
 \end{enumerate}
\end{corollary}

\bibliographystyle{alpha}
\bibliography{citations}

\begin{thebibliography}{BHPVdV04}

\bibitem[AGV72]{SGA4}
M.~Artin, A.~Grothendieck, and J.~Verdier.
\newblock {\em Th\'eorie de topos et cohomologie \'etale des sch\'emas (SGA
  4)}.
\newblock S{\'e}minaire de g{\'e}om{\'e}trie alg{\'e}brique du Bois-Marie -
  1963-64. Springer-Verlag Berlin; New York, 1972.

\bibitem[AM69]{AM}
M.~Artin and B.~Mazur.
\newblock {\em Etale homotopy}.
\newblock Springer-Verlag, Berlin, 1969.

\bibitem[BHPVdV04]{BHPV}
W.~Barth, K.~Hulek, C.~Peters, and A.~Van~de Ven.
\newblock {\em Compact complex surfaces}, volume 3,4 of {\em Ergebnisse der
  Mathematik und ihrer Grenzgebiete}.
\newblock Springer-Verlag, Berlin, second edition, 2004.

\bibitem[CJS09]{2009arXiv0905.2191C}
V.~{Cossart}, U.~{Jannsen}, and S.~{Saito}.
\newblock Canonical embedded and non-embedded resolution of singularities for
  excellent two-dimensional schemes.
\newblock {\em arXiv:0905.2191v2}, 2009.

\bibitem[Deb01]{Deb}
O.~Debarre.
\newblock {\em Higher-dimensional algebraic geometry}.
\newblock Universitext. Springer-Verlag, 2001.

\bibitem[DG70]{SGA3}
M.~Demazure and A.~Grothendieck.
\newblock {\em Sch{\'e}mas en groupes (SGA 3)}.
\newblock S{\'e}minaire de g{\'e}ometrie alg{\'e}brique du Bois Marie -
  1962-64. Springer-Verlag Berlin; New York, 1970.

\bibitem[Fri73]{MR0313254}
E.~M. Friedlander.
\newblock {$K(\pi ,\,1)$}'s in characteristic {$p>0$}.
\newblock {\em Topology}, 12:9--18, 1973.

\bibitem[Fuj02]{Fuj00}
K.~Fujiwara.
\newblock A proof of the absolute purity conjecture (after {G}abber).
\newblock In {\em Algebraic geometry 2000, {A}zumino ({H}otaka)}, volume~36 of
  {\em Adv. Stud. Pure Math.}, pages 153--183. Math. Soc. Japan, Tokyo, 2002.

\bibitem[Ful93]{Ful93}
W.~Fulton.
\newblock {\em Introduction to Toric Varieties}.
\newblock Princeton University Press, 1993.

\bibitem[GR11]{GR11}
O.~Gabber and L.~Ramero.
\newblock F{oundations of $p$-adic Hodge theory}, 2011.
\newblock Fith Release.

\bibitem[Gro71]{SGA1}
A.~Grothendieck.
\newblock {\em Rev{\^e}tements \'etales et groupe fondamental (SGA 1)}.
\newblock S{\'e}minaire de g{\'e}ometrie alg{\'e}brique du Bois Marie -
  1960-61. Springer-Verlag Berlin; New York, 1971.

\bibitem[Hir53]{Hi53}
F.~Hirzebruch.
\newblock \"uber vierdimensionale riemannsche fl\"achen mehrdeutiger
  analytischer funktionen von zwei komplexen ver\"anderlichen.
\newblock {\em Mathematische Annalen}, 126:1--22, 1953.

\bibitem[H{\"u}b16]{Dissertation}
K.~H{\"u}bner.
\newblock {\em {Aspherical neighborhoods on arithmetic surfaces}}.
\newblock PhD thesis, 2016.

\bibitem[JS12]{JS12}
U.~Jannsen and S.~Saito.
\newblock Bertini theorems and {L}efschetz pencils over discrete valuation
  rings, with applications to higher class field theory.
\newblock {\em J. Algebraic Geom.}, 21(4):683--705, 2012.

\bibitem[Kat94]{Kato94}
K.~Kato.
\newblock Toric singularities.
\newblock {\em American Journal of Mathematics}, 116:1075--1099, 1994.

\bibitem[Lip78]{MR0491722}
J.~Lipman.
\newblock Desingularization of two-dimensional schemes.
\newblock {\em Annals of Mathematics}, 107:151--207, 1978.

\bibitem[Liu02]{Liu}
Q.~Liu.
\newblock {\em Algebraic geometry and arithmetic curves}, volume~6 of {\em
  Oxford Graduate Texts in Mathematics}.
\newblock Oxford University Press, Oxford, 2002.

\bibitem[N{\'e}r52]{Ne52}
A.~N{\'e}ron.
\newblock Probl{\`e}mes arithm{\'e}tiques et g{\'e}om{\'e}triques rattach{\'e}s
  {\`a} la notion de rang d'une courbe alg{\'e}brique dans un corp.
\newblock {\em Bulletin de la Soci{\'e}t{\'e} Math{\'e}matique de France},
  80:101--166, 1952.

\bibitem[Poo04]{Poo04}
B.~Poonen.
\newblock Bertini theorems over finite fields.
\newblock {\em Annals of Mathematics}, 160:1099--1127, 2004.

\bibitem[Sch07]{Schmidt07}
A.~Schmidt.
\newblock Rings of integers of type $k(\pi,1)$.
\newblock {\em Documenta Mathematica}, 12:441--471, 2007.

\bibitem[Sch10]{MR2629694}
A.~Schmidt.
\newblock {\"{U}ber pro-{$p$}-Fundamentalgruppen markierter arithmetischer
  Kurven}.
\newblock {\em Journal f\"ur die Reine und Angewandte Mathematik. [Crelle's
  Journal]}, 640:203--235, 2010.

\bibitem[Vie77]{Vie77}
E.~Viehweg.
\newblock {Invarianten der degenerierten Fasern in lokalen Familien von
  Kurven}.
\newblock {\em Journal f\"ur die reine und angewandte Mathematik (Crelles
  Journal)}, pages 284--308, 1977.

\end{thebibliography}

\end{document}